\documentclass{article}
\usepackage[utf8]{inputenc}
\usepackage[utf8]{inputenc}
\usepackage{amsmath}
\usepackage{amsthm}
\usepackage{amssymb}
\usepackage{amsfonts}
\usepackage{authblk}
\usepackage{mathtools}
\usepackage{stmaryrd}
\usepackage{tikz-cd}
\usepackage{enumerate}
\usepackage{bbm}
\usepackage{hyperref}
\hypersetup{colorlinks=true,
    citecolor=[rgb]{0.153,0.255,0.545},
    urlcolor=black,
    colorlinks=true,
    linkcolor=[rgb]{0.153,0.255,0.545}}
\usepackage{bbm}
\usepackage{bbm}
\usepackage{comment}

\usepackage[top=1in, bottom=1.25in, left=1.25in, right=1.25in]{geometry}

\newtheorem{theorem}{Theorem}[section]
\newtheorem{introthm}{Theorem}

\newtheorem*{theorem*}{Theorem}
\newtheorem{proposition}[theorem]{Proposition}
\newtheorem{lemma}[theorem]{Lemma}
\newtheorem{corollary}[theorem]{Corollary}
\theoremstyle{definition}
\newtheorem{example}[theorem]{Example}
\newtheorem{remark}[theorem]{Remark}

\newtheorem{definition}[theorem]{Definition}

\renewcommand{\AA}{\mathbb{A}}
\newcommand{\GG}{\mathbb{G}}
\newcommand{\FF}{\mathbb{F}}

\newcommand{\QQ}{\mathbb{Q}}
\newcommand{\UU}{\mathbb{U}}

\newcommand{\ZZ}{\mathbb{Z}}
\newcommand{\CC}{\mathbb{C}}

\newcommand{\AGL}{\textup{AGL}}
\newcommand{\GL}{\textup{GL}}
\newcommand{\SL}{\textup{SL}}
\newcommand{\BG}{{\textup{B}G}}

\newcommand{\cat}[1]{\textup{\bfseries #1}}

\newcommand{\Mod}{\cat{Mod}}
\newcommand{\Vect}{\cat{Vect}}
\newcommand{\Bord}{\cat{Bord}}

\newcommand{\Stck}{\cat{Stck}}
\newcommand{\RStck}{\cat{RStck}}
\newcommand{\Var}{\cat{Var}}
\newcommand{\Ab}{\cat{Ab}}
\newcommand{\LG}{\cat{LG}}
\newcommand{\Grpd}{\cat{Grpd}}
\newcommand{\FinGrpd}{\cat{FinGrpd}}
\newcommand{\FGGrpd}{\cat{FGGrpd}}
\newcommand{\Mnfd}{\cat{Mnfd}}

\newcommand{\id}{\textup{id}}
\newcommand{\K}{\textup{K}_0}
\newcommand{\Span}{\textup{Span}}
\newcommand{\Hom}{\textup{Hom}}
\newcommand{\Fr}{\textup{Fr}}

\DeclareMathOperator{\tr}{tr}
\DeclareMathOperator{\Spec}{Spec}

\DeclareMathOperator{\Aut}{Aut}

\newcommand{\stI}{\ooalign{\hidewidth $\int$\hidewidth\cr\rule[0.7ex]{1ex}{.4pt}}}

\newcommand{\bdscale}{0.5}

\newcommand{\bdunit}[1][1]{
\begin{tikzpicture}[semithick, scale=#1*\bdscale, baseline=-0.5ex]
\begin{scope}
    \draw (0,0) ellipse (0.2cm and 0.4cm);
    \draw (0,-0.4) arc (-90:90:0.75cm and 0.4cm);
\end{scope}
\end{tikzpicture}
}

\newcommand{\bdcounit}[1][1]{
\begin{tikzpicture}[semithick, scale=#1*\bdscale, baseline=-0.5ex]
\begin{scope}
    \draw (0,0) ellipse (0.2cm and 0.4cm);
    \draw (0,0.4) arc (90:270:0.75cm and 0.4cm);
\end{scope}
\end{tikzpicture}
}

\newcommand{\bdgamma}[1][1] {
\begin{tikzpicture}[semithick, scale=#1*\bdscale, baseline=-0.5ex]
\begin{scope}
    \draw (0,0.5) ellipse (0.2cm and 0.4cm);
    \draw (0,-0.5) ellipse (0.2cm and 0.4cm);
    \draw (0,0.9) arc (90:-90:1.1cm and 0.9cm);
    \draw (0,0.1) arc (90:-90:0.3cm and 0.1cm);
\end{scope}
\end{tikzpicture}
}

\newcommand{\bdbeta}[1][1] {
\begin{tikzpicture}[semithick, scale=#1*\bdscale, baseline=-0.5ex]
\begin{scope}
    \draw (0,0.5) ellipse (0.2cm and 0.4cm);
    \draw (0,-0.5) ellipse (0.2cm and 0.4cm);
    \draw (0,0.9) arc (90:270:1.1cm and 0.9cm);
    \draw (0,0.1) arc (90:270:0.3cm and 0.1cm);
\end{scope}
\end{tikzpicture}
}

\newcommand{\bdgenus}[1][1]{
\begin{tikzpicture}[semithick, scale=#1*\bdscale, baseline=-0.5ex]
\begin{scope}
    \draw (-1,0) ellipse (0.2cm and 0.4cm);
    \draw (-1,0.4) .. controls (-0.5,0.6) and (0.5,0.6) .. (1,0.4);
    \draw (-1,-0.4) .. controls (-0.5,-0.6) and (0.5,-0.6) .. (1,-0.4);
    \draw (-0.5,0.1) .. controls (-0.5,-0.125) and (0.5,-0.125) .. (0.5,0.1);
    \draw (-0.4,0.0) .. controls (-0.4,0.0625) and (0.4,0.0625) .. (0.4,0.0);
    \draw (1,0) ellipse (0.2cm and 0.4cm);
\end{scope}
\end{tikzpicture}
}

\newcommand{\bdcomultiplication}[1][1]{
\begin{tikzpicture}[semithick, scale=#1*\bdscale, baseline=-0.5ex]
\begin{scope}
    \draw (-1,0.5) ellipse (0.2cm and 0.4cm);
    \draw (-1,-0.5) ellipse (0.2cm and 0.4cm);
    \draw (1,0) ellipse (0.2cm and 0.4cm);
    \draw (-1,0.9) .. controls (0,0.9) and (0,0.4) .. (1,0.4);
    \draw (-1,-0.9) .. controls (0,-0.9) and (0,-0.4) .. (1,-0.4);
    \draw (-1,0.1) .. controls (-0.2,0.1) and (-0.2,-0.1) .. (-1,-0.1);
\end{scope}
\end{tikzpicture}
}

\newcommand{\bdmultiplication}[1][1]{
\begin{tikzpicture}[semithick, scale=#1*\bdscale, baseline=-0.5ex]
\begin{scope}
    \draw (-1,0) ellipse (0.2cm and 0.4cm);
    \draw (1,0.5) ellipse (0.2cm and 0.4cm);
    \draw (1,-0.5) ellipse (0.2cm and 0.4cm);
    \draw (-1,0.4) .. controls (0,0.4) and (0,0.9) .. (1,0.9);
    \draw (-1,-0.4) .. controls (0,-0.4) and (0,-0.9) .. (1,-0.9);
    \draw (1,0.1) .. controls (0.2,0.1) and (0.2,-0.1) .. (1,-0.1);
\end{scope}
\end{tikzpicture}
}

\newcommand{\bdtwist}[1][1] {
\begin{tikzpicture}[semithick, scale=#1*\bdscale, baseline=-0.5ex]
\begin{scope}
    \draw (-1,0.5) ellipse (0.2cm and 0.4cm);
    \draw (-1,-0.5) ellipse (0.2cm and 0.4cm);
    \draw (1,0.5) ellipse (0.2cm and 0.4cm);
    \draw (1,-0.5) ellipse (0.2cm and 0.4cm);
    \draw (-1,0.9) .. controls (0,0.9) and (0,-0.1) .. (1,-0.1);
    \draw (-1,0.1) .. controls (0,0.1) and (0,-0.9) .. (1,-0.9);
    \draw (-1,-0.9) .. controls (0,-0.9) and (0,0.1) .. (1,0.1);
    \draw (-1,-0.1) .. controls (0,-0.1) and (0,0.9) .. (1,0.9);
\end{scope}
\end{tikzpicture}
}

\title{\Large \textbf{Arithmetic-Geometric Correspondence of Character Stacks \\ via Topological Quantum Field Theory}}
\author[$\dagger$]{\'Angel Gonz\'alez-Prieto}
\author[$\ddagger$]{Márton Hablicsek}
\author[$\star$]{Jesse Vogel}
\affil[$\dagger$]{\footnotesize Department of Algebra, Geometry and Topology, Universidad Complutense de Madrid, Plaza de Ciencias 3, Ciudad Universitaria, 28040 Madrid, Spain; and Instituto de Ciencias Matem\'aticas (CSIC-UAM-UCM-UC3), C/ Nicol\'as Cabrera 13-15, 28049 Madrid, Spain, angelgonzalezprieto@ucm.es}
\affil[$\ddagger$]{\footnotesize Department of Mathematics, Leiden University, Niels Bohrweg 1, 2333 CA Leiden, Netherlands, hablicsekhm@math.leidenuniv.nl}
\affil[$\star$]{\footnotesize Department of Mathematics, Leiden University, Niels Bohrweg 1, 2333 CA Leiden, Netherlands, j.t.vogel@math.leidenuniv.nl}
\date{}

\begin{document}

\maketitle

\begin{abstract}
    In this paper, we introduce Topological Quantum Field Theories (TQFTs) generalizing the arithmetic computations done by Hausel and Rodr\'iguez-Villegas and the geometric construction done by Logares, Muñoz, and Newstead to study cohomological invariants of $G$-representation varieties and $G$-character stacks. We show that these TQFTs are related via a natural transformation that we call the `arithmetic-geometric correspondence' generalizing the classical formula of Frobenius on the irreducible characters of a finite group. We use this correspondence to extract some information on the character table of finite groups using the geometric TQFT, and vice versa, we greatly simplify the geometric calculations in the case of upper triangular matrices by lifting its irreducible characters to the geometric setting.
\end{abstract}

\section{Introduction}

In 1896, Frobenius  \cite{frobenius1896gruppencharaktere} introduced
a beautiful formula (for $g = 1$) relating the number of group homomorphisms $\rho : \pi_1(\Sigma_g) \to G$ from the fundamental group of a closed orientable surface $\Sigma_g$ of genus $g$ into a finite group $G$ with the irreducible characters $\hat{G}$ of $G$,
\begin{equation}
\label{eq:frobenius_formula-intro}
|\Hom(\Sigma_g, G)| = |G|\sum_{\chi \in \hat{G}} \left( \frac{|G|}{\chi(1)} \right)^{2g - 2}. 
\end{equation}
More than a century later, Hausel and Rodr\'iguez-Villegas \cite{hausel2008mixed} breathed new life into this formula using a result by Katz that shows these counts can be used to compute Hodge-theoretic invariants associated to a complex variety known as the representation variety.

To be precise, given an algebraic group $G$ and a compact connected smooth manifold $M$, the collection of group homomorphisms from the fundamental group $\pi_1 (M)$ of the manifold into $G$ forms an algebraic variety induced by the algebraic variety structure on $G$,
\[ R_G(M) = \Hom(\pi_1(M), G) \]
called the $G$-representation variety of $M$. Explicitly, in the case of a closed orientable surface $\Sigma_g$ of genus $g$, the $G$-representation variety is the closed subvariety of $G^{2g}$ given by
\begin{equation}
    \label{eq:presentation_representation_variety_closed_surface}
    R_G(\Sigma_g) = \left\{ (A_1, B_1, \ldots, A_g, B_g) \in G^{2g} \;\bigg|\; \prod_{i = 1}^{g} [A_i, B_i] = 1 \right\} .
\end{equation}
Questions about the geometry of the $G$-representation variety have attracted researchers in various fields of mathematics, as the theory combines the algebraic geometry of the group $G$, the topology of $M$, and the group theory and representation theory of both $\pi_1 (M)$ and $G$ (see, for instance, \cite{cooper1994plane,cooper1998representation, culler1983varieties,dunfield2004non,munoz2016geometry}). There are several successful frameworks to study cohomological invariants of representation and character varieties, the main two representatives are the following.

\begin{itemize}
    \item \emph{The arithmetic method}: It is the method developed by Hausel and Rodr\'iguez-Villegas to relate the complex structure of the $G(\CC)$-representation variety with the point count of the $G(\FF_q)$-representation varieties over finite fields $\FF_q$. This method strongly relies on a result by Katz, inspired by the Weil conjectures, that shows that if the point count of the varieties over $\FF_q$ is a certain polynomial in $q$, then such polynomial is the so-called $E$-polynomial of the associated complex variety, a polynomial collecting alternating sums of the Hodge numbers in the spirit of an Euler characteristic. In this way, using Frobenius' formula (\ref{eq:frobenius_formula-intro}), in \cite{hausel2008mixed} the authors managed to compute the $E$-polynomials of $\GL_n(\CC)$-character varieties in purely combinatorial terms. Subsequent works have extended this method to other scenarios such as $G = \SL_r(\CC)$ \cite{mereb2015polynomials}, $G=\GL_r(\CC)$ with a generic parabolic structure \cite{mellit2020poincare}, non-orientable surfaces \cite{letellier2020series}, and connected split reductive groups \cite{bridger2022character}.

\item \emph{The geometric method}: It was started by Logares, Mu\~noz, and Newstead in the seminar paper \cite{logmunnew13}. The goal of this method is to compute cohomological invariants of the representation variety by cutting it into smaller, simpler pieces whose cohomological invariants can be easily computed. From them, they use motivic techniques to compile this local information and assemble it for the whole variety. Thanks to this approach, more subtle invariants can be computed, such as the virtual class of the representation variety in the Grothendieck ring $\K(\Var_k)$ of algebraic varieties. Its drawback is that it requires very involved calculations that can only be conducted for particular examples \cite{marmun16,mar17}.
\end{itemize}

The aim of this work is to show that these two completely unrelated methods can be actually encompassed into a common framework, by means of Topological Quantum Field Theories (TQFTs). Furthermore, we shall push Frobenius' formula further in several directions to bring new striking connections between point counting, the character theory of finite groups, and the motivic theory of complex manifolds.

The use of TQFTs to understand representation varieties has been proven to be very successful. In \cite{gonlogmun20}, the first version of a TQFT computing the virtual Hodge structure of the complex representation variety was constructed. This construction has been extended several times to virtual classes in the Grothendieck ring of varieties \cite{gon20}, for manifolds with singularities \cite{gonzalez2023character}, and for character stacks \cite{gonzalez2022virtual}. However, instead of the usual category $\Bord_n$ of $n$-dimensional bordisms, all these constructions are forced to work in the weaker category of `pointed' bordisms, with an arbitrary choice of basepoints on it.
With this method, the motivic classes of the representation varieties of surface groups were computed for $\SL_2(\CC)$ and for groups of upper triangular matrices \cite{gonlogmun20, hablicsek2022virtual, vogel2023motivic}. Notice that, with this method, in contrast to the arithmetic method, not just the $E$-polynomial of the representation variety can be computed, but much more, its class in the Grothendieck ring of varieties computing its universal cohomology theory.

It is worth mentioning that the study of the geometry of representation varieties is actually intrinsically tied to mathematical physics through the so-called {$G$-character varieties}. In the case that the group $G$ is reductive, they are defined using Geometric Invariant Theory \cite{mumford1994geometric} as the GIT quotient of the representation variety $R_G(M)$ with respect to the adjoint action of $G$ given by conjugation
$$
    \mathcal{M}_G(M) = R_G(M) \sslash G.
$$
The character variety parametrizes representations $\pi_1(M) \to G$ up to isomorphism, in contrast to the representation variety which only parametrizes raw representations. 

These character varieties are widely studied in mirror symmetry \cite{hausel2003mirror,mauri2021topological}, in string theory \cite{diaconescu2018bps}, and in the context of Yang-Mills and Chern-Simons quantum field theories, as in Witten's quantization of the Jones polynomial of knots \cite{witten1991quantum}. 
Furthermore, when $M = \Sigma_g$ is a compact orientable surface of genus $g$, character varieties are also called Betti moduli spaces, and they are one of the three main moduli spaces studied in non-Abelian Hodge Theory, alongside the Dolbeault moduli space parametrizing semistable $G$-Higgs bundles on a smooth projective curve $C$ of genus $g$, and the de Rham moduli space parametrizing flat $G$-connections on the curve $C$. Interestingly, with some assumptions, these moduli spaces are smooth varieties with canonically identified underlying differentiable manifolds via the Riemann-Hilbert, Hitchin-Kobayashi and Simpson correspondences. However, their algebraic structures and their cohomological invariants are completely different: for instance, the celebrated and recently resolved $P=W$ conjecture states the weight filtration on the cohomology of the character variety coincides with the perverse filtration on the cohomology of the Dolbeault moduli space \cite{de2022hitchin, de2012topology, hausel2022p, maulik2022p, mauri2022geometric}.

\paragraph{Arithmetic results.}

Let us discuss the results obtained in this work in the arithmetic setting first. Fix a finite group $G$, and let us focus on the point count of the $G$-representation variety for a compact manifold $M$, $|R_G(M)| = |\Hom(\pi_1(M), G)|$.

Our first result in this direction is that this point count can be quantized by means of a Topological Quantum Field Theory (TQFT). A TQFT is a symmetric monoidal functor
\[ Z: \Bord_n \to \Vect_k \]
from the category of $n$-dimensional bordisms to the category of $k$-vector spaces (or, more generally, the category $R\textup{-}\Mod$ of modules over a ring $R$). Roughly speaking, the functor $Z$ transforms closed $(n-1)$-dimensional manifolds $M$ into a vector space $Z(M)$ and bordisms $W$ between manifolds $M_1$ and $M_2$ into linear maps $Z(W): Z(M_1) \to Z(M_2)$.

An important use of these functors is to compute invariants. For instance, suppose that we are interested in a certain invariant $\chi(W)$ of closed connected $n$-dimensional manifolds $W$ with values in a field $k$. Such closed manifolds can be seen as bordisms $W: \varnothing \to \varnothing$ and since $Z(\varnothing) = k$, under a TQFT they correspond to $k$-linear maps $Z(W): k \to k$, that is, multiplication by an element $Z(W)(1) \in k$.
If we manage to construct a TQFT such that this element $Z(W)(1) = \chi(W)$ is precisely the invariant we want to study, we can use this gadget to decompose $W$ into simpler pieces, and to use the functoriality of $Z$ to re-assemble the invariant $\chi(W)$ from simpler algebraic information. In this situation, we shall say that $Z$ quantizes the invariant $\chi$. For instance, in the case of the genus $g$ surface $\Sigma_g$, we can decompose it as
\[ \Sigma_g = \bdcounit \circ \bdgenus^g \circ \bdunit \]
and hence the invariant $\chi(\Sigma_g)$ can be computed through the composition of three linear maps
\begin{equation}\label{eq:decomposition-TQFT-surface}
    k = Z(\varnothing) \xleftarrow{Z \left(\bdcounit\right)} Z(S^1) \xleftarrow{Z\left(\bdgenus\right)^g} Z(S^1) \xleftarrow{Z\left(\bdunit\right)} Z(\varnothing) = k : Z(\Sigma_g) .
\end{equation}

This is, in abstract terms, the central idea of the geometric method developed by Logares, Mu\~noz and Newstead.

Our first result in this direction is to formally construct a TQFT quantizing the point count of the $G$-representation variety.

\begin{introthm}[Proposition \ref{prop:invariant-arithmetic}]\label{thm:arithmetic-tqft-intro}
For any finite group $G$ and any $n \geq 1$, there exists a TQFT
\[ Z^\#_G : \Bord_n \to \Vect_\CC \]
such that $Z^\#_G(W)(1) = |R_G(W)|/|G|$ for any closed connected $n$-dimensional manifold $W$.
\end{introthm}

This TQFT has a natural interpretation in terms of characters of the group $G$ that connects directly with Frobenius' formula (\ref{eq:frobenius_formula-intro}). It is well-known \cite{kock} that $2$-dimensional TQFTs with values in $\Vect_\CC$ are in one-to-one correspondences with commutative Frobenius algebras over $\CC$. %
The correspondence is given by assigning to the TQFT $Z$ the vector space $A = Z(S^1)$, which inherits a natural commutative Frobenius algebra structure from the pair-of-pants bordisms.

In the case of the TQFT of Theorem \ref{thm:arithmetic-tqft-intro}, the associated Frobenius algebra is exactly $Z^\#_G(S^1) = R_\CC(G)$, the representation ring of $G$ generated by the characters of $G$, with convolution of class functions as product and the usual inner product of characters as bilinear form, as proven in Theorem \ref{prop:equiv-frobenius}. In this way, $Z^\#_G\left(\bdgenus[0.7]\right)$ is an endomorphism of $R_\CC(G)$ whose eigenvectors are precisely the irreducible characters $\chi \in \hat{G}$ with eigenvalues $|G|^2/\chi(1)^2$. At the light of these observations, we realize that Frobenius' formula (\ref{eq:frobenius_formula-intro}) is nothing but a direct corollary of the decomposition (\ref{eq:decomposition-TQFT-surface}).

\paragraph{Geometric results.}

Regarding the geometric techniques, our aim is to understand the geometry of the representation variety $R_G(M)$ for an algebraic group $G$ and a compact manifold $M$. A key accomplishment of this work in this direction is a definition of character stack in a functorial way. In the literature, the $G$-character stack of a manifold $M$ is usually defined as the quotient stack $[R_G(M) / G]$ of the $G$-representation variety mod out by the action of $G$ by conjugation. However, this definition has a critical problem that prevents it to generalize well to gluings: The representation variety $R_G(M) = \Hom(\pi_1(M), G)$ depends on the fundamental group of $M$, which is not a functor out of the category of topological spaces, but out of the category of \emph{pointed} topological spaces. This is the reason why all the previously studied TQFT, like \cite{gonlogmun20,gonzalez2023character,gonzalez2022virtual} needed to deal with pointed bordisms, adding a piece of extra spurious information.

To address this problem, instead of considering representations of the fundamental group of $M$, we shall consider representations $\rho: \Pi(M) \to G$ of the fundamental groupoid $\Pi(M)$ of $M$, mod out by the action of the `local gauge' group $\mathcal{G}_M = \prod_{x \in M} G$, given by $(\overline{g} \cdot \rho)(\gamma) = g_y \; \rho(\gamma) \; g_x^{-1}$ for a path $\gamma$ joining $x, y \in M$ and $\overline{g} = (g_x)_{x \in M} \in \mathcal{G}_M$. In this way, we define the $G$-character stack of $M$ as the quotient stack
$$
    \mathfrak{X}_G(M) = [\Hom(\Pi(M), G) / \mathcal{G}_M].
$$
Notice that, in principle, such quotient is not well-defined since both $\Hom(\Pi(M), G)$ and $\mathcal{G}_M$ are infinite-dimensional. However, we will show that we can provide a precise meaning to the above expression by considering finitely generated groupoids equivalent to $\Pi(M)$ (see Corollary \ref{cor:fin-gen-character-stack} and the discussion therein).

Since we are now using fundamental groupoids instead of fundamental groups, this construction inherits much better continuity properties. For instance, the functor $\mathfrak{X}_G$ sends colimits into limits. In particular, if $W \cup_M W'$ is the gluing of two compact manifolds $W$ and $W'$ along a common boundary $M$, we have $\mathfrak{X}_G(W \cup_M W') = \mathfrak{X}_G(W) \times_{\mathfrak{X}_G(M)} \mathfrak{X}_G(W')$. In this direction, the main result we obtain in this work is that we can use the continuity of the local gauge character stack to actually construct a TQFT computing the virtual class $[\mathfrak{X}_G(W)]\in \K(\Stck_k)$ in the Grothendieck ring $\K(\Stck_k)$ of algebraic stacks with affine stabilizers. To the best of our knowledge, this is the first genuine TQFT constructed in the literature computing these virtual classes in the unpointed setting.

\begin{introthm}[Theorem \ref{thm:character_stack_TQFT}]\label{thm:geometric-tqft-intro}
For any algebraic group $G$ and any $n \geq 1$, there exists a lax TQFT
$$
    Z_G: \Bord_n \to \K(\Stck_k)\textup{-}\Mod
$$
quantizing the virtual class of the $G$-character stack.
\end{introthm}

Notice that this TQFT is not monoidal, but only lax monoidal, that is, there exists a morphism $Z_G(M_1)\otimes Z_G(M_2) \to Z_G(M_1 \sqcup M_2)$ but this morphism is not an isomorphism. This implies in particular that $Z_G(M)$ may not be finitely generated for an object $M$ (and, in general, it is not so), and thus, the 
classification theorem of $2$-dimensional TQFTs does not apply. In this sense, the algebra $Z_G(S^1) = \K(\Stck_{[G/G]})$ should be seen as the infinite-dimensional analogue of the representation ring for general algebraic groups $G$. It is worth mentioning that the result of Theorem \ref{thm:geometric-tqft-intro} is sharp since the no-go theorem \cite[Theorem 4.21]{gonzalez2023quantization} proves that no genuinely monoidal TQFT exists for this purpose.

\paragraph{The arithmetic-geometric correspondence.}

Apart from the results obtained in this paper in the arithmetic and geometric setting, the main aim of this work is precisely to relate both constructions. The first result in this direction is that, in fact, we have a natural transformation connecting the geometry with the arithmetics.

\begin{introthm}[Corollary \ref{cor:natural_transformation_geometric_arithmetic}]
    \label{thm:cover-arithmetic-geometric}
    Let $G$ be a connected algebraic group over a finite field $\FF_q$. Then there exists a natural transformation $Z_G \Rightarrow Z^\#_{G(\FF_q)}$ as functors $\Bord_n \to \Ab$.
\end{introthm}

It is worth mentioning that the assumption of $G$ being connected is crucial here, since otherwise such natural transformation may not exist. This is related to the fact that for non-connected $G$ the $\FF_q$-points $\mathfrak{X}_G(M)(\FF_q)$ of the character stack is not the same groupoid as the action groupoid of $G(\FF_q)$ on the representation variety $R_{G(\FF_q)}(M)$, as it already happens for $G = \ZZ / 2\ZZ$. In some sense, the geometric TQFT is studying the former groupoid, whereas the arithmetic TQFT is studying the latter.

However, we can go even a step further and encompass both constructions under a general TQFT by using the so-called Landau-Ginzburg (LG) models. Roughly speaking, an LG model over a stack $\mathfrak{S}$ is a pair $(\mathfrak{X}, f : \mathfrak{X} \to k)$ of a representable $\mathfrak{S}$-stack $\mathfrak{X}$ together with a scalar function $f$ on its closed points. Using them and mimicking the construction of Theorem \ref{thm:geometric-tqft-intro}, we obtain the following result. 

\begin{introthm}[Theorem \ref{thm:tqft_LG}] For any algebraic group $G$ and any $n \geq 1$, there exists a lax monoidal TQFT
\begin{equation*}
    Z^{\LG}_G : \Bord_n \rightarrow \K(\LG_k)\textup{-}\Mod
\end{equation*}
computing the virtual class of the character stack in the Grothendieck ring $\K(\LG_k)$ of Landau-Ginzburg models over $k$.
\end{introthm}
 
This TQFT encodes both the arithmetic and the geometric methods via natural transformations.  
For this purpose, we consider the forgetful, inclusion, and `integration along the fibers' maps, which are respectively ring morphisms
$$
\stI: \K(\LG_\mathfrak{S})\rightarrow \K(\Stck_\mathfrak{S}), \qquad \iota: \K(\Stck_\mathfrak{S}) \rightarrow \K(\LG_\mathfrak{S}) , \qquad \int: \K(\LG_\mathfrak{S}) \rightarrow \CC^{\mathfrak{S}(k)}.
$$
Here $\mathfrak{S}$ is a fixed stack and $\CC^{\mathfrak{S}(k)}$ denotes the ring of (set-theoretic) functions on the objects of the groupoid $\mathfrak{S}(k)$ which are invariant under isomorphism. The forgetful map $\stI$ just operates by sending the class $[(\mathfrak{X} \to \mathfrak{S}, f)]$ of a Landau-Ginzburg model to the virtual class $[\mathfrak{X} \to \mathfrak{S}]$ of its underlying $\mathfrak{S}$-stack, whereas the inclusion map $\iota$ sends the class of a stack $[\mathfrak{X} \to \mathfrak{S}]$ into the Landau-Ginzburg model $[(\mathfrak{X} \to \mathfrak{S}, 1_{\mathfrak{X}})]$, where $1_{\mathfrak{X}}$ denotes the constant $1$ function. Furthermore, if $k$ is a finite field, the `integration along the fibres' map $\int$ sends a Landau-Ginzburg model $[(\pi: \mathfrak{X} \to \mathfrak{S}, f)]$ to the function $\mathfrak{S}(k) \to \CC$ given by $s \mapsto \sum_{x\in \pi^{-1}(s)} f(x)$.

Using these maps, we can obtain the following result, which we call the 
\emph{arithmetic-geometric correspondence} and represents the main theorem of this paper. In the following, we take $G$ to be any algebraic group defined over a base ring $R$ (typically $R=\ZZ$) and $\FF_q$ is any finite field `extending' $R$, in the sense that there exists a ring homomorphism $R \to \FF_q$.

\begin{introthm}\label{introthm:nattrans}
For any connected algebraic group $G$, there exist natural transformations
\begin{equation*}
    \begin{tikzcd}[row sep=0.5em] 
        & Z^\LG_G\arrow[rd,shift left=2, "\stI",bend left,Rightarrow] \arrow[Rightarrow]{ld}[swap]{\int} & \\
        Z^{\#}_{G(\FF_q)}  & & Z_G\arrow[bend left, Rightarrow]{lu}[shift left]{i}
    \end{tikzcd}
\end{equation*} 
\end{introthm}

As an application of these results, we use this correspondence in two ways. First, assuming that the eigenvectors of point-counting TQFT $Z_{G(\FF_q)}^{\#}$ lift to the Landau-Ginzburg TQFT, the eigenvectors of the geometric TQFT $Z_G$ can be easily computed. We apply this idea in the case of $G$ being the group of unitriangular matrices and we significantly simplify the TQFT used in \cite{hablicsek2022virtual} computing the virtual classes of representation varieties corresponding to the group of unitriangular matrices.

Second, assuming that a universal Landau-Ginzburg TQFT exists, we derive information about the character table of the family of finite groups $G(\FF_q)$. To do so, we need to impose an extra finiteness condition. Recall that, since $Z_G^{\LG}$ is not monoidal, $Z_G^{\LG}(S^1) = \K(\LG_{[G/G]})$ is not in general finitely generated as $\K(\LG_k)$-module. Let us denote by $[\BG]_{[G/G]}$ the virtual class of the Landau-Ginzburg model $(\BG \to [G / G], 1_{\BG})$ over $[G / G]$. We shall say that $\mathcal{V} \subseteq Z_G^{\LG}(S^1)$ is a `core submodule' if it is a finitely generated $\K(\LG_k)$-module such that $[\BG]_{[G/G]} \in \mathcal{V}$ and $Z_G\left(\bdgenus[0.7]\right)(\mathcal{V}) \subseteq \mathcal{V}$.

Recall we have a natural identification of the ring of iso-invariant functions $\CC^{[G/G](\FF_q)}$ with the representation ring $R_{\CC}(G(\FF_q))$. In this way, the `integration along fibers' map provides morphisms
\[ \int_{R} : \K(\LG_{k}) \to \CC, \qquad \int_{[G/G]} : \K(\LG_{[G/G]}) \to R_{\CC}(G(\FF_q)) . \]

\begin{introthm}[Corollary \ref{cor:relation_eigenvalues_geometric_arithmetic}]
    \label{thm:applications-intro}
    Let $G$ be an algebraic group over a finitely generated $\ZZ$-algebra $R$, and let $\FF_q$ be a finite field extending $R$. Denote by $\textup{\bfseries 1}_G \in \K(\Stck_{[G/G]})$ the class of the inclusion of the identity in $G$.
    If the $\K(\Stck_k)$-module $\mathcal{V} = \langle Z_G\left(\bdgenus[0.7]\right)^g(\textup{\bfseries 1}_G) \textup{ for } g = 0, 1, 2, \ldots \rangle$ is finitely generated, then:
    \begin{enumerate}[(i)]
        \item The dimensions of the complex irreducible characters of $G(\FF_q)$ are precisely given by
        \[ d_i = \frac{|G(\FF_q)|}{\sqrt{\lambda_i}} \]
        for $\lambda_i \in \ZZ$ the eigenvalues of $\int_{R} A$, where $A$ is any matrix representing the linear map $Z_G\left(\bdgenus[0.7]\right)$ with respect to a generating set of $\mathcal{V}$.
        \item Write $\int_{R} \textup{\bfseries 1}_G = \sum_i v_i$, where $v_i$ are eigenvectors of $\int_{R} A$ with eigenvalues $\lambda_i$. Then
        \[ v_i =\frac{d_i}{|G(\FF_q)|} \sum_{\substack{\chi \in \hat{G} \textup{\ s.t.} \\ \chi(1) = d_i}} \chi \]
        for every $i$.
    \end{enumerate}
\end{introthm}

Notice that this is a remarkable result since it relates the character theory of the infinite family $G(\FF_q)$ of finite groups of Lie type with a single linear map associated to the geometry of the stack $[G/G]$. Indeed, suppose that, in the case of base field $k = \overline{\QQ}$, we manage to compute a set of eigenvectors $\mathfrak{X}_1, \ldots, \mathfrak{X}_m$ of the core submodule $\mathcal{V} \subseteq Z_G^{\LG}(S^1)$. For this purpose, since we are in the characteristic zero algebraically closed setting, we can apply `heavy' technology such as Weyl's complete reducibility theorem and Jordan forms for the elements of $G$. However, since for this computation we use only finitely many morphisms and stacks, the same calculation works over a certain finitely generated ring $R$. Therefore, the eigenvectors $\mathfrak{X}_1, \ldots, \mathfrak{X}_m$ and their eigenvalues can be used in Theorem \ref{thm:applications-intro} to provide information of the character table of $G(\FF_q)$ for any finite field $\FF_q$ extending $R$. We illustrate this strategy with a couple of examples in which we can fully determine the character table of some families of finite groups from the simpler calculation of the TQFT over a characteristic zero algebraically closed field. 

\paragraph{Structure of the document.}

The paper is organized as follows. In Section \ref{sec:representation_ring} we describe a $2$-dimensional TQFT by means of the representation ring of a finite group, and its connection to Frobenius' formula. In Section, \ref{sec:character-stack-tqft} we develop the theory of local gauge character stacks and use it to construct a TQFT computing their virtual classes. In Section \ref{sec:arithmetic_tqft} we discuss the arithmetic counterpart of the previous TQFT, and we prove it to be isomorphic to the one induced by the representation ring (Section \ref{sec:arithmetic-tqft-frobenius}) and to be covered by the geometric TQFT (Section \ref{sec:arithmetic-tqft-character-stack}) in the case of connected groups. The Landau-Ginzburg theory and its corresponding TQFT are described in Section \ref{sec:tqftLG}. We finish the paper with some applications of the arithmetic-geometric correspondence, as discussed in Section \ref{sec:app}.

\textbf{Acknowledgements.}
The first-named author thanks the hospitality of the University of Leiden, where part of this work was developed during a research stay. The second and third authors thank the hospitality of the Universidad Complutense de Madrid, where the initial stage of this work was done. The second author thanks Professor Kremnitzer for fruitful conversations about TQFTs.
The first-named author has been partially supported by Severo Ochoa excellence programme SEV-2015-0554, Spanish Ministerio de Ciencia e Innovaci\'on project PID2019-106493RB-I00, and by the Madrid Government (\textit{Comunidad de Madrid} – Spain) under the Multiannual
Agreement with the Universidad Complutense de Madrid in the line Research Incentive for
Young PhDs, in the context of the V PRICIT (Regional Programme of Research and Technological Innovation) through the project PR27/21-029.

\section{Representation ring and Frobenius algebras}
\label{sec:representation_ring}

Let $G$ be a finite group, and denote by $A = R_\CC(G)$ the representation ring of $G$, that is, the complex algebra generated by $\CC$-valued class functions on $G$. Of great importance in the representation theory of $G$ is the inner product that is defined on $A$, which we will denote by $\beta : A \otimes_\CC A \to \CC$, and it is given by
\[ \beta(a \otimes b) = \frac{1}{|G|} \sum_{g \in G} a(g) b(g^{-1}) \quad \textup{ for } a, b \in A . \]
A lesser known but equally important operation on $A$ is the convolution operation $\mu : A \otimes_\CC A \to A$ on $A$, which is given by
\[ \mu(a \otimes b)(g) = \sum_{h \in G} a(h) b(h^{-1} g) \quad \textup{ for } a, b \in A , \]
and is related to the inner product via $\beta(a \otimes b) = \mu(a \otimes b)(1)$ for $a, b \in A$. The unit $\eta : \CC \to A$ with respect to $\mu$ is given by $\eta(1) = \mathbbm{1}_1$, where $\mathbbm{1}_1$ is the characteristic function of $1$, that is, $\mathbbm{1}_1(1) = 1$ and $\mathbbm{1}_1(g) = 0$ for $g \ne 1$. We write $\eta$ for this unit, rather than $1$, to not confuse it with the constant class function $1$ which is the unit for the usual point-wise multiplication on $A$. Alternatively, $\eta$ can be expressed as
\[ \eta(1) = \frac{1}{|G|} \sum_{\chi \in \hat{G}} \chi(1) \, \chi , \]
where $\hat{G}$ denotes the set of complex irreducible characters of $G$. To complete our setup, we will also consider the map $\gamma : \CC \to A \otimes_\CC A$ given by
\[ \gamma(1) = \sum_{\chi \in \hat{G}} \chi \otimes \chi . \]
Note that $\gamma(1)$ can be seen as an inner product on the conjugacy classes of $G$, as the function $G \times G \to \CC$ is given by the cardinal
\[ \gamma(1)(g_1, g_2) = \left|\left\{ h \in G \mid h g_1 h^{-1} = g_2 \right\}\right| . \]
It turns out these operations give $R_\CC(G)$ the structure of a \textit{Frobenius algebra} over $\CC$.

\begin{definition}
    A \emph{Frobenius algebra} over a field $k$ is an algebra $A$ over $k$ equipped with a bilinear form $\beta : A \otimes_k A \to k$, which is
    \begin{itemize}
        \item associative, that is, $\beta(ab \otimes c) = \beta(a \otimes bc)$ for all $a, b, c \in A$,
        \item non-degenerate, that is, there exists a $k$-linear map $\gamma : k \to A \otimes_k A$ such that $(\beta \otimes \id_A) (a \otimes \gamma(1)) = a = (\id_A \otimes \beta)(\gamma(1) \otimes a)$ for all $a \in A$.
    \end{itemize}
\end{definition}

\begin{proposition}
    The representation ring $R_\CC(G)$ is a commutative Frobenius algebra over $\CC$ with multiplication $\mu$ and bilinear form $\beta$.
\end{proposition}
\begin{proof}
    First note that $\mu$ is associative as
    \[ \begin{aligned} \mu(a \otimes \mu(b \otimes c))(g) &= \sum_{h_1, h_2 \in G} a(h_1) b(h_2) c(h_2^{-1} h_1^{-1} g) \\ &= \sum_{h_1, h_2 \in G} a(h_2) b(h_2^{-1} h_1) c(h_1^{-1} g) = \mu(\mu(a \otimes b) \otimes c)(g) \end{aligned} \]
    for all $a, b, c \in A$ and $g \in G$. Similarly, $\mu$ is commutative as
    \[ \begin{aligned} \mu(a \otimes b)(g) &= \sum_{h \in G} a(h) b(h^{-1}g) = \sum_{h' \in G} a(h'^{-1}g) b(g^{-1}h'g) = \sum_{h' \in G} a(h'^{-1}g) b(h') = \mu(b \otimes a),\end{aligned} \]
    for all $a, b \in A$ and $g \in G$, where in he second equality we set $h' = g h^{-1}$. Furthermore, $\beta$ is associative as
    \[ \beta(\mu(a \otimes b) \otimes c) = \frac{1}{|G|} \sum_{g, h \in G} a(h) b(h^{-1} g) c(g^{-1}) = \frac{1}{|G|} \sum_{g, h \in G} a(g) b(g^{-1} h) c(h^{-1}) = \beta(a \otimes \mu(b \otimes c)) \]
    for all $a, b, c \in A$. Finally, $\beta$ is non-degenerate as $\gamma : \CC \to A \otimes_\CC A$ satisfies
    \[ (\beta \otimes \id_A)(a \otimes \gamma(1)) = \sum_{\chi \in \hat{G}} \beta(a \otimes \chi) \chi = a \]
    by the orthogonality of irreducible characters \cite[Theorem 2.12]{fulton2013representation}, and similarly $(\id_A \otimes \beta)(\gamma(1) \otimes a) = a$, for all $a \in A$. 
\end{proof}

Frobenius algebras naturally carry even more structure: the structure of a coalgebra. In particular, there is a coassociative comultiplication $\delta : A \to A \otimes_\CC A$ with a counit $\varepsilon : A \to \CC$ which are given by
\[ \delta(a) = (\mu \otimes \id_A)(a \otimes \gamma(1)) %
\quad \textup{ and } \quad \varepsilon(a) = \beta(a \otimes 1) = \beta(1 \otimes a) . \]
The algebra and coalgebra structures are compatible in a certain sense, see \cite{kock} for details. In the case of the representation ring $A = R_\CC(G)$, one can check that the comultiplication and counit are explicitly given by
\[ \delta(a) = \sum_{\chi \in \hat{G}} \mu(a \otimes \chi) \otimes \chi \quad \textup{ and } \quad \varepsilon(a) = \frac{1}{|G|} a(1) \quad \textup{ for all } a \in A . \]

There is a surprising relation between the $G$-representation varieties of closed surfaces and the character table of $G$, given by the following proposition, which goes back to Frobenius \cite{frobenius1896gruppencharaktere}. What is remarkable, is that this equation can be written purely in terms of the operations of the Frobenius algebra structure on the representation ring of $G$.

\begin{proposition}[Frobenius' formula, \cite{frobenius1896gruppencharaktere}]
    \label{prop:frobenius_formula}
    The number of points of the $G$-representation variety of $\Sigma_g$ is given by
    \begin{equation}\label{eq:frobenius_formula} \frac{|R_G(\Sigma_g)|}{|G|} = (\varepsilon \circ (\mu \circ \delta)^g \circ \eta)(1) = \sum_{\chi \in \hat{G}} \left( \frac{|G|}{\chi(1)} \right)^{2g - 2} .
    \end{equation}
\end{proposition}

In order to prove this proposition, we will make use of the following lemma.
\begin{lemma}
    \label{lemma:for_frobenius_formula}
    In any Frobenius algebra $A$, one has $\mu \circ \delta = \mu \circ (\id_A \otimes (\mu \circ \delta \circ \eta)(1))$.
\end{lemma}
\begin{proof}
    Let us start with the right-hand side. By definition of $\delta$, this is equal to
    \[ \mu \circ (\id_A \otimes (\mu \circ (\mu \otimes \id_A)(\eta(1) \otimes \gamma(1)))) . \]
    Since $\eta$ is the unit for $\mu$, this reduces to
    \[ \mu \circ (\id_A \otimes (\mu \circ \gamma(1))) . \]
    By coassociativity of $\mu$, this is equal to
    \[ \mu \circ (\mu \otimes \id_A) \circ (\id_A \otimes \gamma(1)) . \]
    Finally, by definition of $\delta$, this is equal to the left-hand side.
\end{proof}

\begin{proof}[Proof of Proposition \ref{prop:frobenius_formula}]
    Let $f : G \to \CC$ be the class function given by $f(g) = |\{ (A, B) \in G^2 \mid [A, B] = g \}|$.
    From the explicit expression of the representation variety \eqref{eq:presentation_representation_variety_closed_surface} and the definition of the convolution $a * b = \mu(a \otimes b)$ on $R_\CC(G)$, it follows immediately that the number of points of the representation variety is given by
    \[ |R_G(\Sigma_g)| = (\underbrace{f * \cdots * f}_{g \textup{ times}})(1) . \]
    Therefore, using Lemma \ref{lemma:for_frobenius_formula}, it is sufficient to show that $f$ is equal to $(\mu \circ \delta \circ \eta)(1) = \sum_{\chi \in \hat{G}} \frac{|G|}{\chi(1)} \chi$,
    or equivalently, that $\beta(f \otimes \chi) = \frac{|G|}{\chi(1)}$ for any irreducible complex character $\chi$ of $G$. Note that
    \[ \beta(f \otimes \chi) = \frac{1}{|G|} \sum_{g \in G} f(g) \chi(g^{-1}) = \frac{1}{|G|} \sum_{A, B \in G} \chi([A, B]^{-1}) = \frac{1}{|G|} \sum_{A, B \in G} \chi(B A B^{-1} A^{-1}) . \]
    Let $\rho : G \to \GL(V)$ be a representation with character $\chi$. From Schur's lemma it follows that, for any $A \in G$, the operator $T_A = \sum_{B \in G} \rho(B A B^{-1})$ is a scalar multiple of the identity, that is, $T_A = \frac{\tr(T_A)}{\chi(1)} =  \frac{|G|}{\chi(1)} \chi(A)$. Hence, it follows that
    \[ \beta(f \otimes \chi) = \frac{1}{|G|} \sum_{A, B \in G} \tr(T_A A^{-1}) = \frac{1}{|G|} \sum_{A \in G} \frac{|G|}{\chi(1)} \chi(A) \chi(A^{-1}) = \frac{|G|}{\chi(1)} . \]
    The second equality follows directly from the definition of $\varepsilon, \mu, \delta$ and $\eta$.
\end{proof}

\subsection{Topological Quantum Field Theories}
\label{sec:tqft}
While Frobenius' formula, Proposition \ref{prop:frobenius_formula}, can be beautifully expressed in terms of the operations on a Frobenius algebra, it is possible to view this formula in an even more general framework, that of \textit{Topological Quantum Field Theories}, or TQFTs for short. In this section, we shall review the main concepts and notations we will use around TQFTs. For a more thorough treatment of TQFTs, see \cite{atiyah1988topological,kock}. Let us recall the definition of an (oriented) bordism.

\begin{definition}
    Let $i : M \to \partial W$ be a smooth embedding of a closed oriented $(n - 1)$-dimensional manifold $M$ into the boundary of a compact oriented $n$-dimensional manifold $W$. Then $i$ is an \emph{in-boundary} (resp.\ \emph{out-boundary}) if for all $x \in M$, positively oriented bases $v_1, \ldots, v_{n - 1}$ for $T_x M$, and $w \in T_{i(x)} W$ pointing inwards (resp.\ outwards) compared to $W$, the basis $di_x(v_1), \ldots, di_x(v_{n - 1}), w$ for $T_{i(x)} W$ is positively oriented.
    Given two closed oriented $(n - 1)$-dimensional manifolds $M_1$ and $M_2$, a \emph{bordism} from $M_1$ to $M_2$ is a diagram
    \[ M_2 \xrightarrow{i_2} W \xleftarrow{i_1} M_1 \]
    consisting of a compact oriented manifold $W$ with boundary $\partial W$, an in-boundary $i_1 : M_1 \to \partial W$ and an out-boundary $i_2 : M_2 \to \partial W$, such that $i_1(M_1) \cap i_2(M_2) = \varnothing$ and $\partial W = i_1(M_1) \cup i_2(M_2)$.
    Two bordisms $(W, i_1, i_2)$ and $(W', i_1', i_2')$ between $M_1$ and $M_2$ are \emph{equivalent} if there exists a diffeomorphism $\alpha : W \to W'$ such that the diagram
    \[ \begin{tikzcd}[row sep=0.5em]
        & W \arrow{dd}{\alpha} & \\ M_2 \arrow{ur}{i_2} \arrow[swap]{dr}{i_2'} & & M_1 \arrow[swap]{ul}{i_1} \arrow{dl}{i_1'} \\ & W' &
    \end{tikzcd} \]
    commutes.
\end{definition}
    
\begin{definition}
    Let $n \ge 1$ be an integer. The \emph{category of $n$-dimensional bordisms}, denoted $\Bord_n$, is the category whose objects are closed oriented $(n - 1)$-dimensional manifolds (possibly empty). A morphism $W : M_1 \to M_2$ between two manifolds is an equivalence class of bordisms between $M_1$ and $M_2$. Composition in $\Bord_n$ is given by gluing along the common boundary, which can be shown to be a well-defined operation on equivalence classes of bordisms \cite{milnor}. Furthermore, the category $\Bord_n$ is endowed with a symmetric monoidal structure given by the disjoint union of both objects and bordisms.
\end{definition}

\begin{remark}
    From Section \ref{sec:field_theory_and_quantization} on, we consider $\Bord_n$ as a $2$-category rather than a $1$-category, where $2$-morphisms are given by equivalences of bordisms. However, for the purposes of this section, it suffices to consider $\Bord_n$ as the `truncated' $1$-category.
\end{remark}

Given a commutative ring $R$, denote by $R\textup{-}\Mod$ the category of $R$-modules. When $R = k$ is a field, this category will also be denoted by $\Vect_k$, the category of $k$-vector spaces. These categories are symmetric monoidal categories with monoidal structure given by tensor product over $R$.

\begin{definition}
    An \emph{$n$-dimensional Topological Quantum Field Theory (TQFT)} is a symmetric monoidal functor
    \[ Z : \Bord_n \to R\textup{-}\Mod . \]
    When $Z$ is not monoidal, but only lax monoidal, $Z$ is called a \emph{lax TQFT}.
\end{definition}

\begin{remark}
    In this context, recall that monoidality for $Z : \Bord_n \to R\textup{-}\Mod$ means that $Z(\varnothing) = R$ and we have natural isomorphisms $Z(M_1 \sqcup M_2) \cong Z(M_1) \otimes_R Z(M_2)$, for objects $M_1$ and $M_2$. Lax monoidality, on the other hand, means that $Z(\varnothing) = R$ and there exists a homomorphism $Z(M_1) \otimes_R Z(M_2) \to Z(M_1 \sqcup M_2)$, but this map may not be an isomorphism.
\end{remark}

TQFTs have an important application in algebraic topology, as they can be used to compute algebraic invariants of closed manifolds.
Suppose that we are interested in a particular $R$-valued invariant $\chi(W)$ associated to a closed $n$-dimensional manifold $W$. For instance, $\chi(W)$ might be some (co)homological information of a moduli space attached to $W$. In that case, we will say that a TQFT $Z : \Bord_n \to R\textup{-}\Mod$ \emph{quantizes} $\chi$ if, for any closed $n$-dimensional manifold $W$, we have that $Z(W)(1) = \chi(W)$. Here, we view $W$ as a bordism $W : \varnothing \to \varnothing$, so that by monoidality of the TQFT, $Z(W)$ will be a morphism $R \to R$ of $R$-modules, that is, $Z(W)$ is simply multiplication by some element of $R$.

\subsection{Frobenius TQFT}
\label{subsec:point_counting_tqft}

TQFTs have been studied in depth in the literature, see for instance \cite{atiyah1988topological,kock}. A key observation in this direction is that the fact that the functor $Z: \Bord_n \to R\textup{-}\Mod$ is monoidal implies that dualizable objects in $\Bord_n$ are sent to dualizable objects in $R\textup{-}\Mod$. In particular, since every object $M$ of $\Bord_n$ is dualizable, $Z(M)$ must be a dualizable object in $R\textup{-}\Mod$, and these are precisely the finitely generated projective modules. This suggests that TQFTs should be characterizable in terms of additional algebraic structures imposed on a finitely generated projective module.

One of the earlier results in this direction came from Dijkgraaf \cite{Dijkgraaf1989}, who observed that $2$-dimensional TQFTs are essentially the same as commutative Frobenius algebras. More precise proofs of this statement were later given by Abrams \cite{abrams1996two} (see also \cite{kock}). 

\begin{theorem}[\cite{Dijkgraaf1989,kock}]
    \label{thm:equivalence_tqft_frobenius}
    Let $k$ be a field. There is an equivalence of categories
    \[ \textbf{\textup{2-TQFTs}} \simeq \textup{\bf CommFrobAlg}_k \]
    between the category of $\Vect_k$-valued $2$-dimensional TQFTs and the category of commutative Frobenius algebras over $k$, given by sending a TQFT to its value on the circle.
\end{theorem}

In Section \ref{sec:representation_ring} we have seen that the representation ring $R_\CC(G)$ of a finite group $G$ carries the structure of a Frobenius algebra. Therefore, it is natural to consider the corresponding $2$-dimensional TFQT, and to ask which invariant of closed surfaces is quantized by this TQFT. Let us denote the TQFT associated to the representation ring $R_\CC(G)$, under the correspondence of Theorem \ref{thm:equivalence_tqft_frobenius}, by
\begin{equation}
    \label{eq:tqft_finite}
    Z_G^\Fr : \Bord_2 \to \Vect_\CC
\end{equation}
which we call the \emph{Frobenius TQFT}. It is explicitly given by $Z_G^\Fr(S^1) = R_\CC(G)$ and
\[ \quad Z_G^\Fr \left( \bdunit \right) = \eta, \quad Z_G^\Fr \left( \bdmultiplication \right) = \mu, \quad Z_G^\Fr \left( \bdbeta \right) = \beta, \]
\[ Z_G^\Fr \left( \bdcounit \right) = \varepsilon, \quad Z_G^\Fr \left( \bdcomultiplication \right) = \delta, \quad Z_G^\Fr \left( \bdgamma \right) = \gamma . \]
\begin{remark}
    The compatibility relations of the algebra and coalgebra structures of a Frobenius algebra translate to compatibilities on the bordisms. We refer the reader to \cite{kock}.
\end{remark}

Naturally, we are now interested in the $\CC$-valued invariant of closed surfaces that this TQFT quantizes. For any surface $\Sigma_g$ of genus $g$, we have
\[ \begin{array}{ccrcccl}\label{eq:frtqftpt}
    Z_G^\textup{\Fr}(\Sigma_g) & = & Z_G^\textup{\Fr} \Bigg( \bdcounit & \circ & \underbrace{\bdgenus \circ \cdots \circ \bdgenus}_{g \textup{ times}} & \circ & \bdunit \Bigg) \\[25pt]
        & = & Z_G^\textup{\Fr}\left(\bdcounit\right) & \circ & Z_G^\textup{\Fr}\left(\bdgenus\right)^g & \circ & Z_G^\textup{\Fr}\left(\bdunit\right) \\[15pt]
        & = & \varepsilon & \circ & (\mu \circ \delta)^g & \circ & \eta .
\end{array} \]

From Proposition \ref{prop:frobenius_formula} it now becomes apparent how to interpret the invariant quantized by this TQFT.

\begin{corollary}
    \label{cor:tqftquanfin}
    The TQFT $Z_G^\textup{\Fr}$ quantizes the number of points of the $G$-representation variety $R_G(\Sigma_g)$ divided by $|G|$. \qed
\end{corollary}

\section{Character stack TQFT}\label{sec:character-stack-tqft}

Let $G$ be an algebraic group over a field $k$. In this section, we shall construct a lax TQFT
\[ Z_G : \Bord_n \to \K(\Stck_k)\textup{-}\Mod \]
computing the virtual classes of character stacks in the Grothendieck ring of algebraic stacks over $k$. This formulation extends and improves the previous version of this functor \cite{gonzalez2022virtual} and sets a proper framework for the non-stacky versions \cite{arXiv181009714,gon20}.

\subsection{Character groupoids}

Before introducing the $G$-character stack, we will forget all geometry, and let $G$ simply be a group.

Recall that a groupoid $\Gamma$ is a category all of whose morphisms are invertible. A groupoid $\Gamma$ is \emph{finitely generated} if it has finitely many objects, all of whose automorphism groups are finitely generated. If, in addition, all automorphism groups are finite we shall say that $\Gamma$ is \emph{finite}. Furthermore, we shall say $\Gamma$ is \emph{essentially finitely generated} (resp.\ \emph{essentially finite}) if it is equivalent to a finitely generated (resp.\ finite) groupoid.

\begin{definition}
    Given a groupoid $\Gamma$, the \emph{$G$-character groupoid} of $\Gamma$ is the groupoid
    \[ \mathfrak{X}_G(\Gamma) = \textbf{Fun}(\Gamma, G) \]
    whose objects are functors $\rho : \Gamma \to G$ (where $G$ is seen as a groupoid with a single object), and morphisms $\rho_1 \to \rho_2$ are given by natural transformations $\mu : \rho_1 \Rightarrow \rho_2$.
    
    The map $\mathfrak{X}_G$ can naturally be extended to a $2$-functor $\mathfrak{X}_G : \Grpd \to \Grpd^\textup{op}$. Explicitly,
    \begin{itemize}
        \item for any functor $f : \Gamma' \to \Gamma$ between groupoids, let $\mathfrak{X}_G(f) : \mathfrak{X}_G(\Gamma) \to \mathfrak{X}_G(\Gamma')$ be the functor given by precomposition $\mathfrak{X}_G(f)(\rho) = \rho \circ f$ for any $\rho \in \mathfrak{X}_G(\Gamma)$, and $\mathfrak{X}_G(f)(\mu) = \mu f$ for any morphism $\mu : \rho_1 \to \rho_2$.
        \item for any natural transformation $\eta : f_1 \Rightarrow f_2$ between functors $f_1, f_2 : \Gamma' \to \Gamma$, let $\mathfrak{X}_G(\eta) : \mathfrak{X}_G(f_1) \Rightarrow \mathfrak{X}_G(f_2)$ be the natural transformation given by $(\mathfrak{X}_G(\eta)_\rho)_{x'} = \rho(\eta_{x'})$ for all $\rho \in \mathfrak{X}_G(\Gamma)$ and $x' \in \Gamma'$. Indeed, this defines a natural transformation as the square
        \[ \begin{tikzcd}
            \rho(f_1(x')) \arrow{r}{\rho(\eta_{x'})} \arrow[swap]{d}{\rho(f_1(\gamma'))} & \rho(f_2(x')) \arrow{d}{\rho(f_2(\gamma'))} \\
            \rho(f_1(y')) \arrow[swap]{r}{\rho(\eta_{y'})} & \rho(f_2(y'))
        \end{tikzcd} \]
        commutes for every $\gamma' : x' \to y'$ in $\Gamma'$ by naturality of $\eta$, and this is natural in $\rho$.
    \end{itemize}
    Note that $\mathfrak{X}_G$ strictly preserves the composition of $1$-morphisms and $2$-morphisms, and therefore defines a strict $2$-functor.
\end{definition}

\begin{corollary}
    \label{cor:equivalent_character_groupoids}
    Any equivalence between groupoids $\Gamma$ and $\Gamma'$ naturally induces an equivalence between the $G$-character groupoids $\mathfrak{X}_G(\Gamma)$ and $\mathfrak{X}_G(\Gamma')$. \qed
\end{corollary}

Now, if $G$ is a finite group and $\Gamma$ is a finitely generated groupoid, then it can easily be seen that $\mathfrak{X}_G(\Gamma)$ is a finite groupoid. Hence, Corollary \ref{cor:equivalent_character_groupoids} implies that $\mathfrak{X}_G(\Gamma)$ is essentially finite for $\Gamma$ essentially finitely generated. Hence, for $G$ a finite group, we can restrict $\mathfrak{X}_G$ to a 2-functor
\[ \mathfrak{X}_G : \FGGrpd \to \FinGrpd^\textup{op} \]
from the 2-category of essentially finitely generated groupoids to (the opposite of) the 2-category of essentially finite groupoids.

\begin{definition}
    Let $M$ be a compact manifold. The \emph{fundamental groupoid} of $M$ is the groupoid $\Pi(M)$ whose objects are the points of $M$, and morphisms $x \to y$ are given by homotopy classes of paths from $x$ to $y$.
\end{definition}

Note that if $M$ is a compact manifold, then $\Pi(M)$ is essentially finitely generated since $M$ is homotopically equivalent to a finite CW-complex \cite{whitehead1940c1}. Moreover, for any smooth map of manifolds $f : M \to N$, there is an induced functor $\Pi(f) : \Pi(M) \to \Pi(N)$. In particular, one can think of $\Pi$ as a functor $\Pi : \Mnfd_c \to \FGGrpd$ out of the category of compact smooth manifolds. Moreover, $\Pi$ can be promoted to a $2$-functor if one considers $\Mnfd_c$ as a $2$-category where $2$-morphisms are given by smooth homotopies.

\begin{definition}
    Let $M$ be a compact manifold. The \emph{$G$-character groupoid} of $M$, denoted $\mathfrak{X}_G(M)$, is defined as $\mathfrak{X}_G(\Pi(M))$. In particular, if $G$ is finite, $\mathfrak{X}_G(M)$ is essentially finite.
\end{definition}

Let us think about the groupoid $\mathfrak{X}_G(M)$ a bit more closely. An object $\rho$ can be identified with a map from the set of homotopy classes of paths on $M$ to $G$ which preserves composition.
A morphism from $\rho_1$ to $\rho_2$ is a natural transformation $\mu : \rho_1 \Rightarrow \rho_2$, which can be thought of as a function $\mu : M \to G$ such that $\rho_2(\gamma) = \mu(y) \rho_1(\gamma) \mu(x)^{-1}$ for any path $\gamma : x \to y$ in $\Pi(M)$. Such transformations are known in physics as \textit{local gauge transformations}.

With this intuition, there is an alternative way to think about the $G$-character groupoid. Defining $\mathcal{G}_\Gamma = \prod_{x \in \Gamma} G$, which we call the \emph{group of local gauge transformations}, it acts on the set $X = \Hom(\Gamma, G)$ by
\[ ((g_x)_{x \in \Gamma} \cdot \rho)(\gamma) = g_y \, \rho(\gamma) \, g_x^{-1} \]
for any $\rho \in X$ and $\gamma : x \to y$ in $\Gamma$. Now the $G$-character groupoid $\mathfrak{X}_G(\Gamma)$ is equivalent to the action groupoid $[X/\mathcal{G}_\Gamma]$. This alternative description will be the key to defining the $G$-character stacks.

\begin{definition}\label{defn:groupoid-cardinality}
    The \emph{groupoid cardinality} of an essentially finite groupoid $\mathfrak{X}$ is
    \[ |\mathfrak{X}| = \sum_{x \in [\mathfrak{X}]} \frac{1}{|\Aut(x)|} \in \QQ , \]
    where $[\mathfrak{X}]$ denotes the quotient category where all the isomorphic objects have been identified and $\Aut(x)$ is the automorphism group of an object $x$ in $[\mathfrak{X}]$ (in other words, $x\in \pi_0(\mathfrak{X})$, i.e, $x$ is an equivalence class of isomorphic objects in $\mathfrak{X}$).
\end{definition}

\begin{proposition}
    \label{prop:count-character-groupoid}
    Let $M$ be a connected compact manifold. Then we have
    \[ |\mathfrak{X}_G(M)| = \frac{|\Hom(\pi_1(M), G)|}{|G|}. \]
\end{proposition}
\begin{proof}
    Since $M$ is connected, its fundamental groupoid is equivalent to the groupoid $\Gamma_M$ with a single object and $\pi_1(M)$ automorphisms. As shown above, in this setting we have that $\mathfrak{X}_G(M)$ is equivalent to the action groupoid $[\Hom(\Gamma_M, G) / \mathcal{G}_{\Gamma_M}]$. This is precisely a groupoid with $\Hom(\Gamma_M, G) = \Hom(\pi_1(M), G)$ objects and the action of $G$ by conjugation. The result follows then from the orbit-stabilizer formula.
\end{proof}

\subsection{Character stacks}
\label{sec:charstacks}

Let $G$ be an algebraic group over a field or, more generally, a finitely generated commutative algebra over $\ZZ$. We construct the $G$-character stack as the geometric analogue of the $G$-character groupoid.

First, we briefly discuss what we mean by a stack in the next couple of paragraphs. Following \cite{eke09}, in this paper, a stack will refer to an algebraic (Artin) stack of finite type over a field (or more generally, over a finitely generated ring over $\ZZ$) with affine stabilizers at every closed point. Similarly, we define the category of stacks, $\Stck_{\mathfrak{S}}$, over a base stack $\mathfrak{S}$ (here, $\mathfrak{S}$ is an algebraic stack, finite type over a field with affine stabilizers) as the slice category of the category of stacks over $\mathfrak{S}$.

\begin{definition}\label{def:stacks}
    The category of stacks over $\mathfrak{S}$, $\Stck_{\mathfrak{S}}$, is defined as follows.
    \begin{itemize}
    \item The objects are pairs $(\mathfrak{X}, \pi)$, where $\mathfrak{X}$ is an algebraic stack of finite type over $k$ with affine stabilizers, and $\pi \colon \mathfrak{X} \to \mathfrak{S}$ is a $1$-morphism of stacks. If the $1$-morphism $\pi$ is understood from the context, we denote the object simply by $\mathfrak{X}$.
    \item A $1$-morphism $(f, \alpha) \colon (\mathfrak{X}, \pi) \to (\mathfrak{X'}, \pi')$ consists of a $1$-morphism of stacks $f \colon \mathfrak{X} \to \mathfrak{X'}$ and a $2$-morphism of stacks $\alpha \colon \pi \Rightarrow \pi' \circ f$.
    \item A $2$-morphism $\mu \colon (f, \alpha) \Rightarrow (g, \beta)$ in $\Stck_\mathfrak{S}$ is a natural isomorphism such that $\pi'(\mu) \circ \alpha = \beta$.
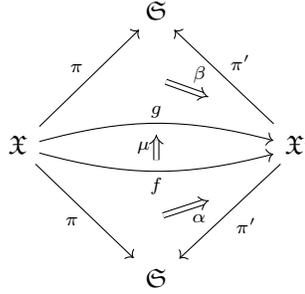
\[ \begin{tikzcd}[column sep=large, row sep = huge]
    & \mathfrak{S} & \\ 
    \mathfrak{X} \arrow[dr, swap, "\pi"{name=U}] \arrow[ur, "\pi"{name=W}] \arrow[rr, swap, bend right=15, "f"{name=A}] \arrow[rr, bend left=15, "g"{name=B}]  \arrow[Rightarrow, shorten=10mm, from=W, to=rr, shift left=1.5ex, "\beta"] \arrow[Rightarrow, shorten = 1.5mm, from=A, to=B, "\mu"] & & \mathfrak{X}' \arrow[dl, "\pi'"]\arrow[ul, swap, "\pi'"] \\
    & \mathfrak{S} \arrow[Rightarrow, swap, shorten=10mm, from=U, to=ru, shift right=2ex, "\alpha"] &
\end{tikzcd} \]
\end{itemize}
\end{definition}

\begin{remark}
    Global quotient stacks of the form $[X/G]$ where $G$ is a linear algebraic group acting on a variety $X$ are examples of algebraic stacks of finite type with affine stabilizers. In fact, any reduced algebraic stack of finite type with affine stabilizers can be stratified by global quotient stacks \cite{kres99}.
\end{remark}

The remark above justifies why we would only work with stacks with affine stabilizers, namely, character stacks are stacks with affine stabilizers. We will use the following lemma about stacks with affine stabilizers several times in the paper.

\begin{lemma}\label{lem:prodaffinestabilizer}
    The category of stacks over $\mathfrak{S}$, $\Stck_{\mathfrak{S}}$, is closed under products. In fact, if $\mathfrak{X}\to \mathfrak{S}$ and $\mathfrak{Y}\to \mathfrak{S}$ are morphisms of stacks with affine stabilizers, then the fiber product $\mathfrak{X}\times_{\mathfrak{S}}\mathfrak{Y}$ also has affine stabilizers.
\end{lemma}

\begin{proof}
    Let $\star$ be a point of the fiber product $\mathfrak{X}\times_{\mathfrak{S}}\mathfrak{Y}$. The induced points on $\mathfrak{X}$, $\mathfrak{Y}$, and $\mathfrak{S}$ are also denoted by $\star$. Furthermore, let us denote the corresponding affine stabilizers by $G_{ \mathfrak{X}\times_{\mathfrak{S}}\mathfrak{Y}}$, $G_\mathfrak{X}$, $G_\mathfrak{Y}$, and $G_\mathfrak{S}$ respectively. Consider the following diagram where all squares (with horizontal and vertical arrows) are Cartesian diagrams.
    \[
    \begin{tikzcd}
        G_{\mathfrak{X}\times_{\mathfrak{S}}\mathfrak{Y}}\ar[r]\ar[d]&G_{\mathfrak{Y}}\ar[r]\ar[d]&\star\ar[dd]\ar[rd,equals]&\\
        G_{\mathfrak{X}}\ar[r]\ar[d]&G_{\mathfrak{S}}\ar[rr]\ar[dd]&& \star\ar[d]\\
        \star\ar[rd,equals]\ar[rr]&&\mathfrak{X}\times_{\mathfrak{S}}\mathfrak{Y}\ar[r]\ar[d]& \mathfrak{X}\ar[d]\\
        &\star\ar[r]&\mathfrak{Y} \ar[r]& \mathfrak{S}
    \end{tikzcd}
    \]
    This diagram shows that the stabilizer of the point $\star$ of $\mathfrak{X}\times_{\mathfrak{S}}\mathfrak{Y}$ is the fiber product of the stabilizers of corresponding points of $\mathfrak{X}$, $\mathfrak{Y}$, and $\mathfrak{S}$, implying that the stabilizer is affine.%
\end{proof}

After our brief introduction of stacks, we turn our attention to defining character stacks. First, we define $G$-representation varieties.

\begin{definition}\label{defn:representation-variety}
    Let $\Gamma$ be a finitely generated groupoid. The \emph{$G$-representation variety} of $\Gamma$ is the variety over $k$ whose functor of points is given by
    \[ R_G(\Gamma)(T) = \Hom(\Gamma, G(T)) , \]
    where the group $G(T)$ is seen as a groupoid with a single object. Functoriality on $T$ is inherited from the functoriality of $G$ and $\Hom$.
\end{definition}
Note that $R_G(\Gamma)$ is indeed representable: choosing generators $\gamma_1, \ldots, \gamma_n$ of $\Gamma$, the $G$-representation variety can be realized as a closed subvariety of $G^n$, and this variety structure can be shown to be independent of the chosen generators. %
However, $R_G(\Gamma)$ is not well-defined up to equivalence of $\Gamma$. This will be fixed once we pass to the $G$-character stack.

Similar to the previous section, we define the group of \emph{local gauge transformations} to be the algebraic group
\[ \mathcal{G}_\Gamma = \prod_{x \in \Gamma} G . \]
There is a natural action of $\mathcal{G}_\Gamma$ on $R_G(\Gamma)$, which is pointwise given by
\[ ((g_x)_{x \in \Gamma} \cdot \rho)(\gamma) = g_y \, \rho(\gamma) \, g_x^{-1} \]
for all $(g_x)_{x \in \Gamma} \in \mathcal{G}_\Gamma(T)$ and $\rho \in R_G(\Gamma)(T)$ and $\gamma : x \to y$ in $\Gamma$.

\begin{definition}
    Let $\Gamma$ be a finitely generated groupoid. The \emph{$G$-character stack} of $\Gamma$ is the quotient stack
    \[ \mathfrak{X}_G(\Gamma) = \left[ R_G(\Gamma) \,/\, \mathcal{G}_\Gamma \right] . \]
\end{definition}

\begin{remark}
    The $G$-character stack is indeed a stack with affine stabilizers (see Definition \ref{def:stacks}) as the stabilizer of a point of $\mathfrak{X}_G(\Gamma)$ is a closed subgroup of $G$. 
\end{remark}

\begin{remark}\label{rem:functorcharstack}
    It is possible to promote $\mathfrak{X}_G(-)$ to a $2$-functor $\FGGrpd \to \Stck_k^\textup{op}$. Let $f : \Gamma' \to \Gamma$ be a functor between finitely generated groupoids. Such a functor induces a morphism between the representation varieties, given by pullback
\[ f^* : R_G(\Gamma) \to R_G(\Gamma'), \quad \rho \mapsto \rho \circ f \quad \textup{ for all } \rho \in R_G(\Gamma)(T) , \]
and also a morphism of groups
\[ \mathcal{G}_f : \mathcal{G}_\Gamma \to \mathcal{G}_{\Gamma'} , \quad (g_x)_{x \in \Gamma} \mapsto (g_{f(x')})_{x' \in \Gamma'} . \]
In particular, there is an induced map on character stacks $\mathfrak{X}_G(\Gamma) \to \mathfrak{X}_G(\Gamma')$ given by sending a $\mathcal{G}_\Gamma$-torsor $P$ to the $\mathcal{G}_{\Gamma'}$-torsor $\mathcal{G}_{\Gamma'} \times_{\mathcal{G}_\Gamma} P$. This is easily seen to be functorial in $f$.

Next, let $\eta : f_1 \Rightarrow f_2$ be a natural transformation between functors $f_1, f_2 : \Gamma' \to \Gamma$. To this natural transformation, we want to assign a $2$-morphism $\mathfrak{X}_G(f_1) \Rightarrow \mathfrak{X}_G(f_2)$, which amounts to, for every $\mathcal{G}_\Gamma$-torsor $P$ over $T$ with $\mathcal{G}_\Gamma$-equivariant map $\rho : P \to R_G(\Gamma)$, a morphism of $\mathcal{G}_{\Gamma'}$-torsors (as indicated by the dashed arrow) such that the diagram
\[ \begin{tikzcd}[row sep=1em, column sep=6em]
    \mathcal{G}_{\Gamma'} \times_{\mathcal{G}_\Gamma} P \arrow[dashed]{dd} \arrow[bend left=10]{dr}{(g', p) \mapsto g' \cdot f_1^*(\rho(p))} & \\
     & R_G(\Gamma') \\
    \mathcal{G}_{\Gamma'} \times_{\mathcal{G}_\Gamma} P \arrow[swap, bend right=10]{ur}{(g', p) \mapsto g' \cdot f_2^*(\rho(p))} &
\end{tikzcd} \]
commutes. Analogous to the case for $G$-character groupoids, this morphism is given by $(g', p) \mapsto (g' \rho(p)(\eta_{x'}), p)$. Note that this map is well-defined (that is, respects the $\mathcal{G}_\Gamma$-action on both sides) by unfolding the definitions.
\end{remark}

\begin{corollary}\label{cor:fin-gen-character-stack}
    Any equivalence between finitely generated groupoids $\Gamma$ and $\Gamma'$ naturally induces an isomorphism between the $G$-character stacks $\mathfrak{X}_G(\Gamma)$ and $\mathfrak{X}_G(\Gamma')$. \qed
\end{corollary}

This observation allows us to extend the definition of the $G$-character stack to groupoids $\Gamma$ which are only essentially finitely generated. In particular, this allows us to define the $G$-character stack of a compact manifold.

\begin{definition}
    Let $M$ be a compact manifold. The \emph{$G$-character stack} of $M$ is defined as
    \[\mathfrak{X}_G(M) = \mathfrak{X}_G(\Gamma) \]
    where $\Gamma$ is any finitely generated groupoid equivalent to the fundamental groupoid $\Pi(M)$ of $M$. By the above corollary, this definition is, up to isomorphism, independent of the choice of $\Gamma$. However, to be exact, one needs to make a choice of $\Gamma$ for every $M$.
\end{definition}

\begin{remark}
    There is a crucial difference between this definition and the one considered in \cite{gonzalez2022virtual}. In the latter, only an action of $G$ on $R_G(\Gamma)$ is considered, given by $(g \cdot \rho) (\gamma) = g^{-1} \rho(\gamma) g$. This can be seen as a ``global'' gauge action, where the coordinate system is changed at the same time at all the vertices. In sharp contrast, in this work, we will consider a \textit{local} gauge action, where we allow to change the gauge differently at each of the vertices of the groupoid.
\end{remark}

\begin{remark}\label{rmk:comparison-stack-groupoid}
    Note that the $G$-character stack could not be described via the functor of points via the $G$-character groupoids since
    \begin{equation}\label{eq:chargroupoidnotfunctor}
        \mathfrak{X}_G(\Gamma)(T) \ne \left[ R_G(\Gamma)(T) \Big/ \mathcal{G}_\Gamma(T) \right] .
    \end{equation}
    For example, let $G = \ZZ / 2\ZZ$ and let $\Gamma$ be the trivial groupoid. Then $\mathfrak{X}_G(\Gamma) = \BG$. However, over a finite field $k$, we have that $\pi_0(\BG(\Spec k))=2$, since there are two principal $\ZZ/2\ZZ$-bundles over $\Spec k$: the trivial bundle and the one corresponding to the degree 2 extension of $k$. On the other hand, $|R_G(\Gamma)(\Spec k)|=1$, showing that Eq, \ref{eq:chargroupoidnotfunctor} does not hold in general.
\end{remark}

\begin{lemma}
    \label{lemma:character_stack_colimits_to_limits}
    $\mathfrak{X}_G(-)$ sends finite colimits in $\FGGrpd$ to finite limits in $\Stck_k$.
\end{lemma}
\begin{proof}
    Since colimits are defined up to equivalence, without loss of generality, we can suppose that we are working in $\FGGrpd$. Let $\Gamma = \operatorname{colim}\limits_{i \in I} \Gamma_i$ be a colimit in $\FGGrpd$. A $T$-point of $\lim_{i \in I} \mathfrak{X}_G(\Gamma_i)$ is a collection of $\mathcal{G}_{\Gamma_i}$-torsors $P_i$ over $T$ with $\mathcal{G}_{\Gamma_i}$-equivariant morphisms $\rho_i : P_i \to R_G(\Gamma_i)(T)$, which are compatible in the sense that there are natural isomorphisms $\mathcal{G}_{\Gamma_i} \times_{\mathcal{G}_{\Gamma_j}} P_j \cong P_i$ in the groupoid $\mathfrak{X}_G(\Gamma_i)(T)$
    for every $i \to j$ in $I$. On the other hand, a $T$-point of $\mathfrak{X}_G(\Gamma)$ is a $\mathcal{G}_\Gamma$-torsor $P$ over $T$ with a $\mathcal{G}_\Gamma$-equivariant morphism $\rho : P \to R_G(\Gamma)(T)$. Note that $\rho$, on $T$-points, is given by
    \[ \rho : P \to R_G(\Gamma)(T) = \Hom(\operatorname{colim}\limits_{i \in I} \Gamma_i, G(T)) = \lim_{i \in I} \Hom(\Gamma_i, G(T)), \]
    so $\rho$ is equivalently described by compatible morphisms $\rho_i : P \to R_G(\Gamma_i)(T)$ which are $\mathcal{G}_{\Gamma_i}$-equivariant, where $\mathcal{G}_{\Gamma_i}$ acts on $P$ via $\mathcal{G}_\Gamma$.

    These two descriptions are related as follows. From the $\mathcal{G}_\Gamma$-torsor $P$, one constructs the $\mathcal{G}_{\Gamma_i}$-torsors $P_i = \mathcal{G}_{\Gamma_i} \times_{\mathcal{G}_\Gamma} P$, which are naturally compatible. Conversely, from the $P_i$ one constructs $\lim_{i \in I} P_i$, where the limit is taken as schemes over $T$, which naturally comes with the structure of a $(\lim_{i \in I} \mathcal{G}_{\Gamma_i})$-torsor, and one puts $P = \mathcal{G}_\Gamma \times_{\left(\lim_{i \in I} \mathcal{G}_{\Gamma_i}\right)} \lim_{i \in I} P_i$.
    This induces the desired isomorphism between $\lim_{i \in I} \mathfrak{X}_G(\Gamma_i)$ and $\mathfrak{X}_G(\Gamma)$.
\end{proof}

\subsection{Grothendieck ring of stacks}\label{sec:groth}

Recall (see Definition \ref{def:stacks}) that all the stacks considered in this paper are algebraic stacks of finite type over a field $k$ (or more generally over a finitely generated ring over $\ZZ$) with affine stabilizers at every closed point.

Following \cite{eke09}, we define the Grothendieck ring of stacks as follows.

\begin{definition}
   Let $\mathfrak{S}$ be an algebraic stack of finite type over a field $k$ with affine stabilizers. We define the \textit{Grothendieck ring of stacks over $\mathfrak{S}$}, denoted by $\K(\Stck_\mathfrak{S})$, as the free abelian group generated by the isomorphism classes $[\mathfrak{X}]$ of stacks $\mathfrak{X}$ over $\mathfrak{S}$, modulo the \textit{scissor relations}
    \[ [\mathfrak{X}] = [\mathfrak{Z}] + [\mathfrak{X} \setminus \mathfrak{Z}], \]
    for every closed substack $\mathfrak{Z} \subset \mathfrak{X}$ with open complement $\mathfrak{X} \setminus \mathfrak{Z}$. Note that $\mathfrak{Z}$ and $\mathfrak{X}\setminus \mathfrak{Z}$ are considered as stacks over $\mathfrak{S}$ via $\mathfrak{X}$. The multiplicative structure is given by the fiber product (see Lemma \ref{lem:prodaffinestabilizer})
    \[ [\mathfrak{X}] \cdot [\mathfrak{Y}] = [ \mathfrak{X} \times_{\mathfrak{S}} \mathfrak{Y} ] , \]
    for any algebraic stacks $\mathfrak{X}$ and $\mathfrak{Y}$ making $\K(\Stck_\mathfrak{S})$ a ring with unit $[(\mathfrak{S}, \id_{\mathfrak{S}})]$ and zero element $[\varnothing]$. Given a stack $\mathfrak{X}$ over $\mathfrak{S}$, its class $[\mathfrak{X}] \in \K(\Stck_\mathfrak{S})$ will be called the \emph{virtual class} of $\mathfrak{X}$.
\end{definition}

\begin{remark}
    In \cite{eke09}, Ekedahl defines a Grothendieck ring of stacks $\widetilde{\K}(\Stck_\mathfrak{S})$ as above with an additional relation that the class of every vector bundle is the same as the trivial bundle, meaning that
    \[ [\mathfrak{E}] = [ \AA^n_{\mathfrak{S}} \times \mathfrak{X} ] \]
    for every vector bundle $\mathfrak{E} \to \mathfrak{X}$ of rank $n$. 
    We choose to omit this additional relation as we would like to keep track of the group action. Of course, there is a natural quotient map (by quotienting out the additional relation)
    \begin{equation}\label{eq:KStck->KStck}
        \K(\Stck_\mathfrak{S})\to \widetilde{\K}(\Stck_\mathfrak{S}).
    \end{equation}
\end{remark}

Ekedahl's version of the Grothendieck ring of stacks over $k$ is isomorphic to the localization of the Grothendieck ring of varieties, $\K(\Var_{k})$, by inverting the class of the affine line $q = [\AA_k^1]$ and the classes of the form $q^n-1$ (coming from the classes of the general linear groups) providing a natural map
\begin{equation}\label{eq:KStck->KVar}
    \widetilde{\K}(\Stck_k)\to \widehat{\K}(\Var_{k})
\end{equation}
where $\widehat{\K}(\Var_{k})$ denotes the completion of the ring $\K(\Var_{k})[q^{-1}]$ with the filtration given by the powers of the class of the affine line.

By Lemma \ref{lem:prodaffinestabilizer}, any morphism $\mathfrak{X} \to \mathfrak{S}$ of algebraic stacks induces a $\K(\Stck_\mathfrak{S})$-module structure on $\K(\Stck_\mathfrak{X})$, where the module structure is given on the generators by
\[ [\mathfrak{T}] \cdot [\mathfrak{Y}] = [\mathfrak{T} \times_\mathfrak{S} \mathfrak{Y}].\]

Similarly, since the property of having affine stabilizers is an absolute notion, any morphism of algebraic stacks $f \colon \mathfrak{X} \to \mathfrak{Y}$ over $\mathfrak{S}$ induces a functor 
\[f_!: \Stck_\mathfrak{X}\to \Stck_\mathfrak{Y}\]
given by composing with $f$ that descends to a $\K(\Stck_\mathfrak{S})$-module morphism
\[\K(\Stck_\mathfrak{X}) \to \K(\Stck_\mathfrak{Y})\]
which we will denote by $f_!$ as well.
Finally, by Lemma \ref{lem:prodaffinestabilizer}, any morphism of algebraic stacks $f \colon \mathfrak{X} \to \mathfrak{Y}$ over $\mathfrak{S}$ induces a functor
\[f^*: \Stck_\mathfrak{Y} \to \Stck_\mathfrak{X}\]
given by the fiber product along $f$. This descends to a $\K(\Stck_\mathfrak{S})$-algebra morphism 
\[ f^* \colon \K(\Stck_\mathfrak{Y}) \to \K(\Stck_\mathfrak{X}).\]

\subsection{Field theory and quantization}
\label{sec:field_theory_and_quantization}

The goal of this section is to construct a lax monoidal TQFT
\[ Z_G : \Bord_n \to \K(\Stck_k)\textup{-}\Mod \]
quantizing the virtual class of $G$-character stacks. This lax TQFT will be constructed as the composition of two functors, the \textit{field theory} and the \textit{quantization functor}. While the field theory will be symmetric monoidal, the quantization functor will only be symmetric lax monoidal.

In fact, our construction will be slightly more general, since we shall construct a lax monoidal functor of $2$-categories, in such a way that the usual functor of $1$-categories appears after collapsing the $2$-morphisms. For this purpose, we need to consider the $2$-categorical version of $\Bord_n$, which we shall also denote by $\Bord_n$. As in the classical setting, the objects of this $2$-category are $(n-1)$-dimensional closed oriented manifolds. A $1$-morphism $(W, i_1, i_2): M_1 \to M_2$ between manifolds $M_1$ and $M_2$ is a bordism between $M_1$ and $M_2$, with boundary inclusions $i_1: M_1 \to W$ and $i_2: M_2 \to W$. Now, a $2$-morphism $\alpha: W \Rightarrow W'$ between two bordisms $(W, i_1, i_2), (W', i_1', i_2'): M_1 \to M_2$ is a boundary-preserving diffeomorphism $\alpha: W \to W'$ such that the diagram
    \[ \begin{tikzcd}[row sep=0.5em]
        & W \arrow{dd}{\alpha} & \\ M_2 \arrow{ur}{i_2} \arrow[swap]{dr}{i_2'} & & M_1 \arrow[swap]{ul}{i_1} \arrow{dl}{i_1'} \\ & W' &
    \end{tikzcd} \]
commutes. Notice that the classical category of bordisms used in Section \ref{sec:tqft} can be obtained from $\Bord_n$ by collapsing $2$-morphisms: we declare two $1$-morphisms to be equivalent if there exists a $2$-morphism between them.

\begin{definition}
    Let $\mathcal{C}$ be a $2$-category with pullbacks. The \emph{$2$-category of spans} over $\mathcal{C}$, denoted $\Span(\mathcal{C})$, is the $2$-category defined as follows. The objects of $\Span(\mathcal{C})$ are the same as the objects of $\mathcal{C}$, and a $1$-morphism between the two objects $X$ and $Y$ is a diagram $X \xleftarrow{f} Z \xrightarrow{g} Y$, called a \emph{span}, in $\mathcal{C}$. A $2$-morphism from a span $X \xleftarrow{f} Z \xrightarrow{g} Y$ to $X \xleftarrow{f'} Z' \xrightarrow{g'} Y$ is a morphism $h : Z \to Z'$ in $\mathcal{C}$ together with $2$-isomorphisms $\alpha : f \to f' \circ h$ and $\beta : g \to g' \circ h$ in $\mathcal{C}$.
    \[ \begin{tikzcd}[row sep=0.5em, column sep=3.0em, execute at end picture={
        \node at (-0.5, 0) {$\Downarrow$};
        \node at (0.5, 0) {$\Downarrow$};
        \node at (-0.75, 0) {\scriptsize $\alpha$};
        \node at (0.75, 0) {\scriptsize $\beta$};
    }] & Z \arrow[swap]{ld}{f} \arrow{rd}{g} \arrow{dd}{h} & \\ X & & Y \\ & Z' \arrow{lu}{f'} \arrow[swap]{ru}{g'} & \end{tikzcd} \]
    Composition of $1$-morphisms is given by pullback: given spans $X \xleftarrow{f} Z \xrightarrow{g} X'$ and $X' \xleftarrow{f'} Z' \xrightarrow{g'} X''$, their composite is the outer span of the diagram
    \[ \begin{tikzcd}[row sep=0.5em]
        & & Z \times_{X'} Z' \arrow{dr} \arrow{dl} & & \\
        & Z \arrow{dr}{g} \arrow[swap]{dl}{f} & & Z' \arrow{dr}{g'} \arrow[swap]{dl}{f'} & \\
        X & & X' & & X''
    \end{tikzcd} \]
    Composition of $2$-morphisms is given by vertical composition.
\end{definition}

\begin{definition}
    \label{def:geometric_field_theory}
    The \emph{geometric field theory} is the $2$-functor
    \[ \mathcal{F}_G : \Bord_n \to \Span(\Stck_k) \]
    defined as follows.
    \begin{itemize}
        \item To any compact oriented $(n - 1)$-dimensional manifold $M$, assign the the $G$-character stack $\mathcal{F}_G(M) = \mathfrak{X}_G(M)$, which is an Artin stack of finite type over $k$ with affine stabilizers.
        \item For any bordism $W : M_1 \to M_2$, the inclusions $i_i : M_i \to W$ induce morphisms of fundamental groupoids $\Pi(M_i) \to \Pi(W)$, which in turn induce morphisms
    \[ \mathfrak{X}_G(M_2) \xleftarrow{i_2^*} \mathfrak{X}_G(W) \xrightarrow{i_1^*} \mathfrak{X}_G(M_1) \]
    as explained in Remark \ref{rem:functorcharstack}. Define $\mathcal{F}_G(W)$ as the span $(\mathfrak{X}_G(W), i_1^*, i_2^*)$.
    \item For any boundary-preserving diffeomorphism $\alpha: W \to W'$ between bordisms $W$ and $W'$, we consider the induced isomorphism $\alpha^*: \mathfrak{X}_G(W') \to \mathfrak{X}_G(W)$ between the associated character stacks. Then, $\mathcal{F}_G(W)$ is the $2$-morphism of $\Span(\Stck_k)$ given by
    \[ \begin{tikzcd}[row sep=0.5em]
         & \mathfrak{X}_G(W)  \arrow{rd} \arrow[swap]{ld} & \\
        \mathfrak{X}_G(M_1)  & & \mathfrak{X}_G(M_1) \\
        & \mathfrak{X}_G(W') \arrow[swap]{uu}{\alpha^*} \arrow[swap]{ru} \arrow{lu} & \\
    \end{tikzcd} \]
\end{itemize}     
\end{definition}

\begin{proposition}
    \label{prop:field_theory_symmetric_monoidal_functor}
    The map $\mathcal{F}_G$ is a symmetric monoidal $2$-functor.
\end{proposition}
\begin{proof}
    The proposition is an adaption of \cite[Proposition 4.7]{gonzalez2022virtual} to a base-point-free setting. In fact, the proof of \cite[Proposition 4.7]{gonzalez2022virtual} goes through line-by-line using the Seifert-van Kampen theorem \cite{brown1967groupoids} and Lemma \ref{lemma:character_stack_colimits_to_limits}.
\end{proof}

Now, let us consider the category $R\textup{-}\Mod$ of modules over a ring $R$. It is naturally a category, which we promote to a $2$-category by considering only identity $2$-morphisms between module homomorphism.

\begin{definition}
    The \emph{quantization functor} is the $2$-functor
    \[ \mathcal{Q} : \Span(\Stck_k) \to \K(\Stck_k)\textup{-}\Mod \]
    defined as follows.
    \begin{itemize}
        \item To an object $\mathfrak{X}$ of $\Span(\Stck_k)$, we associate the $\K(\Stck_k)$-module $\K(\Stck_\mathfrak{X})$.
        \item To a span $\mathfrak{X} \xleftarrow{f} \mathfrak{Z} \xrightarrow{g} \mathfrak{Y}$, we associate the morphism $g_! \circ f^* : \K(\Stck_\mathfrak{X}) \to \K(\Stck_\mathfrak{Y})$ of $\K(\Stck_k)$-modules
    \end{itemize}
\end{definition}

As in \cite{gonlogmun20}, the quantization functor is symmetric lax monoidal, but not monoidal since the natural map
\[ \begin{array}{ccc}
    \K(\Stck_\mathfrak{X}) \otimes_{\K(\Stck_k)} \K(\Stck_\mathfrak{Y}) & \to & \K(\Stck_{\mathfrak{X} \times \mathfrak{Y}}) \\[5pt]
    {[\mathfrak{U} \rightarrow \mathfrak{X}]} \otimes {[\mathfrak{V} \rightarrow \mathfrak{Y}]} & \mapsto & {[ \mathfrak{U} \times \mathfrak{V} \to \mathfrak{X} \times \mathfrak{Y} ]}
\end{array} \]
is generally not an isomorphism. For a proof that $\mathcal{Q}$ is indeed a symmetric lax monoidal functor, see \cite{gonzalez2022virtual}.

\begin{remark}
    Note that the construction of $\mathcal{Q}$ is well-defined, because two spans related by a $2$-morphism are sent to the same module morphisms. Indeed, given a $2$-morphism of spans
    \[ \begin{tikzcd}[row sep=0.5em]
        & \mathfrak{Z} \arrow{dd}{h} \arrow{rd}{g} \arrow[swap]{ld}{f} & \\
        \mathfrak{X} & & \mathfrak{Y}\\
        & \mathfrak{Z}'  \arrow[swap]{ru}{g'} \arrow{lu}{f'} & \\
    \end{tikzcd} \]
    observe that, since $h$ is an isomorphism, the square
    \[ \begin{tikzcd}
        \mathfrak{Z} \arrow[swap]{d}{\id} \arrow{r}{\id} & \mathfrak{Z} \arrow{d}{h} \\
        \mathfrak{Z} \arrow{r}{h} & \mathfrak{Z}'
    \end{tikzcd} \]
    is cartesian, so that $g_! \circ f^* = (g')_! \circ h_! \circ h^* \circ (f')^* = (g')_! \circ \id^* \circ \id_! \circ (f')^* = (g')_! \circ (f')^*$. %
\end{remark}

\begin{definition}
    The \emph{character stack TQFT} $Z_G : \Bord_n \to \K(\Stck_k)\textup{-}\Mod$ is the composition $\mathcal{Q} \circ \mathcal{F}_G$, which is a lax monoidal $2$-functor.
\end{definition}

\begin{remark}   
Since in $R\textup{-}\Mod$ we only allow identities as $2$-morphisms, any $2$-functor $Z: \Bord_n \to R\textrm{-}\Mod$ must send any two diffeomorphic bordisms to the same homomorphism of $ R\textup{-}\Mod$. Hence, any such functor induces a regular functor between the underlying $1$-categories as discussed in Section \ref{sec:tqft}, where in $\Bord_n$ the morphisms are classes of bordisms up to boundary-preserving difeomorphism. In particular, $Z_G: \Bord_n \to \K(\Stck_k)\textup{-}\Mod$ can be also seen as a regular lax monoidal TQFT.
\end{remark}

\begin{theorem}
    \label{thm:character_stack_TQFT}
    Given a closed $n$-dimensional manifold $W : \varnothing \to \varnothing$, the character stack TQFT quantizes the virtual class of the $G$-character stack of $W$, that is,
    \[ Z_G(W)(1) = [\mathfrak{X}_G(W)] \in \K(\Stck_k) . \]
\end{theorem}
\begin{proof}
    The field theory $\mathcal{F}_G(W)$ associated to $W : \varnothing \to \varnothing$ is the span
     \[ \Spec k \xleftarrow{t} \mathfrak{X}_G(W) \xrightarrow{t} \Spec k , \]
     and hence, applying $\mathcal{Q}$, we obtain
     \[ Z_G(W)(1) = t_! t^* (1) = t_! [\mathfrak{X}_G(W)]_{\mathfrak{X}_G(W)} = [\mathfrak{X}_G(W)] \in \K(\Stck_k). \qedhere \]
\end{proof}

\begin{remark}
    Theorem \ref{thm:character_stack_TQFT} generalizes the results of \cite{gonzalez2022virtual} to the basepoint-free setting. To be precise, in \cite{gonzalez2022virtual}, an auxiliary choice of a finite set of points on the bordism is mandatory to make the field theory functorial. However, using the local gauge action explained in Section \ref{sec:charstacks}, instead of the global gauge used in \cite{gonzalez2022virtual}, the results can be extended to the basepoint-free case. 
\end{remark}

\section{Arithmetic TQFT}
\label{sec:arithmetic_tqft}

The goal of this section is to construct the \textit{arithmetic TQFT}, a higher-dimensional analogue of the Frobenius TQFT of Section \ref{sec:representation_ring}. Throughout this section, we will fix a finite group $G$. In applications, this finite group will usually be the $\FF_q$-points of an algebraic group.

While the TQFT of Section \ref{sec:representation_ring} is defined in an ad-hoc manner, in terms of specific operations on the representation ring $R_\CC(G)$, the construction of the arithmetic TQFT will be very similar to that of the character stack TQFT: as the composition of a field theory and a quantization.

\subsection{Field theory and quantization}

\begin{definition}
    \label{def:arithmetic_field_theory}
    The \emph{arithmetic field theory} is the 2-functor $\mathcal{F}^\#_G : \Bord_n \to \operatorname{Span}(\FinGrpd)$ which assigns to a closed $(n - 1)$-dimensional manifold $M$ the $G$-character groupoid
    \[ \mathcal{F}^\#_G(M) = \mathfrak{X}_G(M), \]
    to a 1-morphism, i.e. to a bordism $W : M_1 \to M_2$ the span
    \[ \mathcal{F}^\#_G(W) = \left( \mathfrak{X}_G(M_1) \xleftarrow{i_1^*} \mathfrak{X}_G(W) \xrightarrow{i_2^*} \mathfrak{X}_G(M_2) \right) \]
    and to a 2-morphism, i.e.\ to a diffeomorphism, the 2-cell induced by the diffeomorphism.
\end{definition}

\begin{remark}
    Here we are using the character groupoid $\mathfrak{X}_G(M)$ of $M$ on the finite group $G$, that is, the groupoid $\mathfrak{X}_G(M) = \textbf{Fun}(\Gamma, G)$. When $G$ is the group of $\FF_q$-points of an algebraic group, this groupoid is in general different from the groupoid of $\FF_q$-points of the $G$-character stack, as seen in Example \ref{rmk:comparison-stack-groupoid}.
\end{remark}

\begin{proposition}
    $\mathcal{F}^\#_G$ is a symmetric monoidal functor.
\end{proposition}
\begin{proof}
    Completely analogous to the proof of Proposition \ref{prop:field_theory_symmetric_monoidal_functor}.
\end{proof}

\begin{definition}
    Let $\Gamma$ be a groupoid. Denote by $\CC^\Gamma$ the complex vector space of complex-valued functions on the objects of $\Gamma$ which are invariant under isomorphism. In other words, an element of $\CC^\Gamma$ can be seen as a function $[\Gamma] \to \CC$, where $[\Gamma]$ denotes the set of isomorphism classes of $\Gamma$. 
\end{definition}

Given a functor $f : \Gamma \to \Gamma'$ of groupoids, we can pullback functions via
\[ f^* : \CC^{\Gamma'} \to \CC^\Gamma, \quad \varphi \mapsto \varphi \circ f . \]
Moreover, if $\mu : f \Rightarrow g$ is a natural transformation between functors $f, g : \Gamma \to \Gamma'$, then $f^* = g^*$. In particular, if $\Gamma$ and $\Gamma'$ are equivalent groupoids, then $\CC^\Gamma$ and $\CC^{\Gamma'}$ are naturally isomorphic.

Furthermore, given a functor $f : \Gamma \to \Gamma'$ of essentially finite groupoids, we define push-forward along $f$ as
\[ f_! : \CC^\Gamma \to \CC^{\Gamma'}, \quad \varphi \mapsto \left( \gamma' \mapsto \sum_{[(\gamma, \alpha)] \in [f^{-1}(\gamma')]} \frac{\varphi(\gamma)}{|\Aut(\gamma, \alpha)|} \right) , \]
where $f^{-1}(\gamma')$ denotes the fiber product $\Gamma \times_{\Gamma'} \{ \gamma' \}$ of groupoids.

\begin{definition}
    The \emph{arithmetic quantization functor} is the 2-functor
    \[ \mathcal{Q}^\# : \operatorname{Span}(\FinGrpd) \to \Vect_\CC \]
    which assigns to an essentially finite groupoid $\Gamma$ the vector space $\CC^\Gamma$ and assigns to a span of essentially finite groupoids $\Gamma' \xleftarrow{f} \Gamma \xrightarrow{g} \Gamma''$ the morphism $g_! \circ f^* : \CC^{\Gamma'} \to \CC^{\Gamma''}$.
\end{definition}

\begin{lemma}
    $\mathcal{Q}^\#$ is a symmetric monoidal functor.
\end{lemma}
\begin{proof}
    Let $A' \xleftarrow{f} B \xrightarrow{g} A$ and $A \xleftarrow{h} C \xrightarrow{i} A''$ be two spans of essentially finite groupoids. The relevant diagram in $\Vect_\CC$ is given by
    \[ \begin{tikzcd}[row sep=0.5em]
        & & \CC^{B \times_A C} \arrow{dr}{(\pi_C)_!} & & \\
        & \CC^B \arrow{ur}{\pi_B^*} \arrow{dr}{g_!} & & \CC^C \arrow{dr}{j_!} & \\
        \CC^{A'} \arrow{ur}{f^*} & & \CC^{A} \arrow{ur}{h^*} & & \CC^{A''}
    \end{tikzcd} \]
    where $\pi_B : B \times_A C \to B$ and $\pi_C : B \times_A C \to C$ are the projections. To show $\mathcal{Q}^\#$ preserves compositions, it suffices to show that $h^* \circ g_! = (\pi_C)_! \circ \pi_B^*$.

    First observe that, for any $\tilde{c} \in C$, the groupoids $\pi_C^{-1}(\tilde{c}) = (B \times_A C) \times_C \{ \tilde{c} \}$ and $g^{-1}(h(\tilde{c})) = B \times_A \{ h(\tilde{c}) \}$ are equivalent. Indeed explicitly, an object of $\pi_C^{-1}(\tilde{c})$ is a tuple $(b, c, \alpha, \gamma)$ with $(b, c, \alpha) \in B \times_A C$ and $\gamma : c \to \tilde{c}$ a morphism in $C$. A morphism $(b', c', \alpha', \gamma') \to (b, c, \alpha, \gamma)$ is given by a tuple of morphisms $(\beta : b' \to b, \zeta : c' \to c)$ such that $\alpha \circ g(\beta) = h(\zeta) \circ \alpha'$ and $\gamma \circ \zeta = \gamma'$. By appropriate choice of $\zeta$, this is equivalent to the groupoid whose objects are $(b, \alpha)$ with $b \in B$ and $\alpha : g(b) \to h(\tilde{c})$ and morphisms $(b', \alpha') \to (b, \alpha)$ are morphisms $\beta : b' \to b$ such that $\alpha' \circ g(\beta) = \alpha$. But this is precisely $g^{-1}(h(\tilde{c}))$. 

    Now, for any $\varphi \in \CC^B$ and any $c \in C$, it follows that
    \[ ((\pi_C)_! \pi_B^* \varphi)(c) = \sum_{[(b, c, \alpha, \gamma)] \in [\pi_C^{-1}(c)]} \frac{\varphi(b)}{|\Aut(b, c, \alpha, \gamma)|} = \sum_{[(b, \alpha)] \in [g^{-1}(h(c))]} \frac{\varphi(b)}{|\Aut(b, \alpha)|} = (h^* g_! \varphi)(c) . \qedhere \]
\end{proof}

\begin{definition}
    The \emph{arithmetic TQFT} $Z^\#_G : \Bord_n \to \Vect_\CC$ is the composition $\mathcal{Q}^\# \circ \mathcal{F}^\#_G$.
\end{definition}

\begin{proposition}
    \label{prop:invariant-arithmetic}
    Given a closed $n$-dimensional manifold $W : \varnothing \to \varnothing$, the arithmetic TQFT quantizes the groupoid cardinality of the $G$-character groupoid of $W$, that is,
    \[ Z^\#_G(W)(1) = |\mathfrak{X}_G(W)| = \sum_{[x] \in [\mathfrak{X}_G(W)]} \frac{1}{|\Aut(x)|} , \]
    where the sum runs over the isomorphism classes $[x]$ of the groupoid $\mathfrak{X}_G(W)$.
\end{proposition}
\begin{proof}
    The field theory associated to $W : \varnothing \to \varnothing$ is the span
    \[ \begin{tikzcd} \star & \mathfrak{X}_G(W) \arrow[swap]{l}{c} \arrow{r}{c} & \star \end{tikzcd} \]
    where $\star$ denotes the trivial groupoid.
    Note that $Z^\#_G(\star)$ can naturally be identified with $\CC$. Under this identification, $c^* : \CC \to \CC^{\mathfrak{X}_G(W)}$ is given by $c^*(1) = 1_{\mathfrak{X}_G(W)}$, where $1_{\mathfrak{X}_G(W)}$ denotes the constant function one. Hence,
    \[ Z^\#_G(W)(1) = c_! c^*(1) = c_!(1_{\mathfrak{X}_G(W)}) = \sum_{[x] \in [\mathfrak{X}_G(W)]} \frac{1_{\mathfrak{X}_G(W)}(x)}{|\Aut(x)|} = \sum_{[x] \in [\mathfrak{X}_G(W)]} \frac{1}{|\Aut(x)|} , \]
    where in the third equality we used that $c^{-1}(\star) = \mathfrak{X}_G(W)$.
\end{proof}

\begin{remark}
The whole construction of this section can be repeated considering $\QQ$-valued functions, leading to a functor $\Bord_n \to \Vect_\QQ$. However, as we shall show later, we want to emphasize the similarities of this construction with group characters, and for this reason, it is preferable to consider complex coefficients.
\end{remark}

\subsection{Comparison with the Frobenius TQFT}
\label{sec:arithmetic-tqft-frobenius}

Let us return to the case $n = 2$. For a finite group $G$, we have $\mathfrak{X}_G(S^1) = [G/G]$ where $G$ acts on itself by conjugation. In particular, $Z^\#_G(S^1)$ is the complex vector space of complex functions on $G$ which are invariant under conjugation, which can be naturally identified with the vector space of the representation ring $R_\CC(G)$. That is, there is an isomorphism
\begin{equation}
    \label{eq:isomorphism_representation_ring_equivariant_functions_G}
    Z_G^\#(S^1) = \CC^{[G/G]} \cong R_\CC(G) = Z_G^\Fr(S^1)
\end{equation}

\begin{proposition}
    \label{prop:equiv-frobenius}
    Let $G$ be a finite group. For $n = 2$, there is a natural isomorphism $Z^\#_G \cong Z_G^\Fr$ from the arithmetic TQFT to the Frobenius TQFT of Section \ref{sec:representation_ring}.
\end{proposition}
\begin{proof}
Since both $Z^\#_G$ and $Z_G^\Fr$ are TQFTs, by Theorem \ref{thm:equivalence_tqft_frobenius}, it is enough to compute the Frobenius algebra associated to $Z^\#_G(S^1)$ and to check that it is isomorphic to $R_\CC(G)$ through \ref{eq:isomorphism_representation_ring_equivariant_functions_G}. From (\ref{eq:isomorphism_representation_ring_equivariant_functions_G}), we know the underlying vector spaces are isomorphic. For the rest of the structure, we have
\begin{itemize}
    \item Ring structure. It is given by the image of $W = \bdmultiplication[0.7]$. The field theory $\mathcal{F}^\#_G(W)$ is given by
        \[ \begin{tikzcd}[column sep=4em] {[G/G]^2} & {[G^2 / G]} \arrow[swap]{l}{\pi_1 \times \pi_2} \arrow{r}{m} & {[G/G]} \end{tikzcd} \]
        where $m : G \times G \to G$ is the multiplication map of $G$. Hence, the morphism $Z^\#_G(W) : R_\CC(G) \otimes_\CC R_\CC(G) \to R_\CC(G)$ is given by $a \otimes b \mapsto m_! (\pi_1 \times \pi_2)^* (a \otimes b)$ which a direct check shows that it is precisely $\mu(a \otimes b)$. 
    \item Bilinear form. It is given by $W = \bdbeta[0.7]$. The field theory is given by
    \[ \begin{tikzcd}[column sep=4em] {[G/G]^2} & {[G^2 / G]} \arrow[swap]{l}{\pi_1 \times \pi_2} \arrow{r}{c} & \star, \end{tikzcd} \]
    where $c$ is the terminal map. Hence, the morphism $Z^\#_G(W) : R_\CC(G) \otimes_\CC R_\CC(G) \to \CC$ is given by $a \otimes b \mapsto c_! (\pi_1 \times \pi_2)^* (a \otimes b)$ which is precisely $\beta(a \otimes b) = \mu(a \otimes b)(1)$. 
    \item Unit. It is given by the image of $\bdunit[0.7]$. The corresponding field theory is given by
    \[ \begin{tikzcd}[column sep=4em] {[G/G]} & {[\BG]} \arrow[swap]{l}{e} \arrow{r}{c} & \star \end{tikzcd} \]
    where the map $e$ is induced by the embedding of the identity element $\star\to G$. Hence, the corresponding morphism $Z^\#_G(W):\CC\to R_\CC(G)$ is the unit $\eta$. \qedhere
\end{itemize}
\end{proof}

\begin{remark}
    The use of Theorem \ref{thm:equivalence_tqft_frobenius} can be avoided in Proposition \ref{prop:equiv-frobenius} if instead we check that under the natural identification $Z_G^\#(S^1) = R_\CC(G)$, we have that the images of the generating bordisms
    \[ \bdunit, \qquad \bdcounit, \qquad \bdmultiplication, \qquad \bdcomultiplication, \qquad \bdtwist, \]
    coincide with the ones of $Z_G^\Fr$. This can be checked through a straightforward calculation similar to the ones above.
\end{remark}

\begin{remark}
    Putting together Propositions \ref{prop:equiv-frobenius}, \ref{prop:invariant-arithmetic} and \ref{prop:count-character-groupoid}, we can re-interpret the left-hand side of Frobenius formula (\ref{eq:frobenius_formula}) in a more natural way. Indeed, the mysterious quotient $|R_G(\Sigma)|/|G|$ should be understood as the point count of the character groupoid $|\mathfrak{X}_G(\Sigma)|$.
\end{remark}

\subsection{Comparison with the character stack TQFT}
\label{sec:arithmetic-tqft-character-stack}

In this section, we prove Theorem \ref{introthm:nattrans}. Throughout this section, fix an algebraic group $G$ defined over a finitely generated ring $R$ (typically, $R = \ZZ$ for $G = \GL_n, \SL_n$). Furthermore, let $\FF_q$ be a finite field \emph{extending} $R$, that is, fix a ring morphism $R \to \FF_q$. Such a morphism allows us to talk about the $\FF_q$-rational points of any stack $\mathfrak{X}$ over $R$, denoted $\mathfrak{X}(\FF_q)$.

Now, the bridge from the geometric to the arithmetic world is given by counting $\FF_q$-rational points.

\begin{definition}\label{defn:integral-fibers}
    For any stack $\mathfrak{X}$ over $R$, define the map
    \[ \int_\mathfrak{X} : \K(\Stck_\mathfrak{X}) \to \CC^{\mathfrak{X}(\FF_q)}, \quad [\mathfrak{Y} \xrightarrow{f} \mathfrak{X}] \mapsto (x \mapsto |f^{-1}(x)|) , \]
    where $|f^{-1}(x)|$ denotes the groupoid cardinality of $f^{-1}(x)$ as in Definition \ref{defn:groupoid-cardinality}. %
     Note that $\int_\mathfrak{X}$ can be thought of as integration along the fibers. This map is easily seen to be a ring homomorphism. 
\end{definition}

In particular, note that the map $\int_R : \K(\Stck_R) \to \CC^{R(\FF_q)} = \CC$ induces a restriction-of-scalars functor $(\int_R)^* : \Vect_\CC \to \K(\Stck_R)\textup{-}\Mod$. This allows us to compare the geometric and arithmetic quantization functors.

\begin{proposition}\label{prop:comparisontqft_arith}
    The morphisms $\int_\mathfrak{X}$ define a natural transformation
    \[ \int : \mathcal{Q} \Rightarrow \left(\int_R\right)^* \circ \mathcal{Q}^\# \circ (-)(\FF_q) . \]
    In particular, this induces a natural transformation of TQFTs
    \[ \int : Z_G \Rightarrow \left(\int_R\right)^* \circ \mathcal{Q}^\# \circ (-)(\FF_q) \circ \mathcal{F}_G. \]
    
\end{proposition}
\begin{proof}
    For any span $\mathfrak{X} \xleftarrow{f} \mathfrak{Z} \xrightarrow{g} \mathfrak{Y}$ of stacks over $R$, the relevant diagram of $\K(\Stck_R)$-modules is:
    \[ \begin{tikzcd}
        \K(\Stck_\mathfrak{X}) \arrow{r}{f^*} \arrow{d}{\int_\mathfrak{X}} & \K(\Stck_\mathfrak{Z}) \arrow{r}{g_!} \arrow{d}{\int_\mathfrak{Z}} & \K(\Stck_\mathfrak{Y}) \arrow{d}{\int_\mathfrak{Y}} \\
        \CC^{\mathfrak{X}(\FF_q)} \arrow{r}{f^*} & \CC^{\mathfrak{Z}(\FF_q)} \arrow{r}{g_!} & \CC^{\mathfrak{Y}(\FF_q)}
    \end{tikzcd} \]
    To show that the left square commutes, let $\mathfrak{T} \xrightarrow{j} \mathfrak{X}$ be a morphism of stacks and consider the cartesian diagram:
    \begin{equation*}\label{di:cartesianinstacks}
        \begin{tikzcd}
        \mathfrak{T}\times_{\mathfrak{X}}\mathfrak{Z} \arrow{r} \arrow{d}{h} & \mathfrak{T}\arrow{d}{j} \\
        \mathfrak{Z} \arrow{r}{f} & \mathfrak{X}
    \end{tikzcd}
    \end{equation*}
    Then, for any $\FF_q$-point $z \in \mathfrak{Z}(\FF_q)$ of $\mathfrak{Z}$, we find
    \[ \int_\mathfrak{Z}(f^*\mathfrak{T})(z) = |h^{-1}(z)| = |j^{-1}(f(z))| = f^* \left(\int_{\mathfrak{X}} \mathfrak{T} \right)(z) . \]
    To show that the right square commutes, let $\mathfrak{T} \xrightarrow{j} \mathfrak{Z}$ be a morphism of stacks. Then, for any $\FF_q$-point $y \in \mathfrak{Y}(\FF_q)$ of $\mathfrak{Y}$, we find
    \[ g_!\left(\int_\mathfrak{Z} \mathfrak{T}\right)(y) = \sum_{[(z, \alpha)] \in [g^{-1}(y)]} \frac{|j^{-1}(z)|}{|\Aut(z, \alpha)|} = |(g \circ j)^{-1}(y)| = \int_\mathfrak{Y}(g_!\mathfrak{T})(y) . \qedhere \]
\end{proof}

Next, let us show how the geometric and arithmetic field theories are related. Crucially, we need to assume that $G$ is connected.

\begin{proposition}
    \label{prop:compGconn}
    Suppose $G$ is connected. Then there is a natural isomorphism
    \[ (-)(\FF_q) \circ \mathcal{F}_G \cong \mathcal{F}^\#_{G(\FF_q)} . \]
    In particular, this induces a natural isomorphism of TQFTs
    \[ \mathcal{Q}^\# \circ (-)(\FF_q) \circ \mathcal{F}_G \cong Z^\#_{G(\FF_q)} . \]
\end{proposition}
\begin{proof}
    Since the field theories $\mathcal{F}_G$ and $\mathcal{F}^\#_{G(\FF_q)}$, see Definition \ref{def:geometric_field_theory} and \ref{def:arithmetic_field_theory}, are constructed completely analogous, it suffices to give a natural isomorphism between the groupoid of $\FF_q$-rational points of the character stack $\mathcal{F}_G(M)(\FF_q) = \mathfrak{X}_G(M)(\FF_q)$ and the character groupoid $\mathcal{F}_G^{\#}(M) = \mathfrak{X}_{G(\FF_q)}(M)$. Since both field theories are monoidal, we may assume that $M$ is connected, so $\mathfrak{X}_{G}(M) = [R_G(M) / G]$.
    In this case, the objects of $\mathfrak{X}_G(M)(\FF_q)$ are $G$-principal bundles $P \to \Spec(\FF_q)$ with an equivariant map $P \to R_G(M)$. However, by Lang's theorem \cite{Lang1956} for connected groups, every $G$-principal bundle over $\Spec(\FF_q)$ is trivial, so $P = \Spec(\FF_q) \times G$ and the equivariant map is the same as a map $\Spec(\FF_q) \to R_G(M)$. Hence, the objects of $\mathfrak{X}_G(M)(\FF_q)$ are $\Hom(\Spec(\FF_q), R_G(M)) = R_G(M)(\FF_q) = R_{G(\FF_q)}(M)$, by Definition \ref{defn:representation-variety}. The morphisms between these objects are given by the action of $G(\FF_q)$ on $R_{G(\FF_q)}(M)$ by conjugation. But this is exactly the action stack $[R_{G(\FF_q)}(M) / G(\FF_q)]$, which is precisely the character groupoid $\mathfrak{X}_{G(\FF_q)}(M)$.
\end{proof}

\begin{example}
     If $G$ is not connected, there cannot even be a natural transformation $(-)(\FF_q) \circ \mathcal{F}_G \Rightarrow \mathcal{F}^\#_{G(\FF_q)}$. For instance, consider the $2$-sphere $S^2$ as a bordism $\varnothing \to \varnothing$. On one hand, since $S^2$ is simply connected we have that $\mathcal{F}^\#_{G(\FF_q)}(S^2)$ is the action groupoid $\mathfrak{X}_{G(\FF_q)}(\star) = [\star / G(\FF_q)]$, which is a groupoid with a single object and $|G(\FF_q)|$ automorphisms. On the other hand, the character stack of $S^2$ is $\mathfrak{X}_G(S^2) = \BG$, so $((-)(\FF_q) \circ \mathcal{F}_G)(S^2) = \BG(\FF_q)$. These two groupoids are not in general equivalent, as for $G = \ZZ / 2\ZZ$.
\end{example}

\begin{corollary}\label{cor:natural_transformation_geometric_arithmetic}
    Suppose $G$ is connected. Then there is a natural transformation of TQFTs
    \[ Z_G \Rightarrow \left(\int_R\right)^* \circ Z^\#_{G(\FF_q)} . \qedhere \]
\end{corollary}

The situation can be summarized by the following diagram.
\begin{equation}
    \label{eq:diagram_natural_transformations}
    \begin{tikzcd}[execute at end picture={
        \node at (1.5, 0) {$\Downarrow$};
        \node at (-2.0, 0) {$\Downarrow\!\wr$};
    }]
        & \operatorname{Span}(\Stck_k) \arrow{r}{\mathcal{Q}} \arrow{dd}{(-)(\FF_q)} & \K(\Stck_k)\textup{-}\Mod \\
        \Bord_n \arrow{ur}[name=FG]{\mathcal{F}_G} \arrow[swap]{dr}[name=FGk]{\mathcal{F}^\#_{G(\FF_q)}} & & \\
        & \operatorname{Span}(\FinGrpd) \arrow[swap]{r}{\mathcal{Q}^\#} & \Vect_\CC \arrow[swap]{uu}{\left(\int_R\right)^*}
    \end{tikzcd}
\end{equation}

From these results, we observe that for $G$ non-connected, the two functors $Z_G^\# = \mathcal{Q}^\# \circ \mathcal{F}_{G(\FF_q)}^\#$ and $\mathcal{Q}^\# \circ (-)(\FF_q) \circ \mathcal{F}_{G}$ are different TQFTs. For further reference, we shall give a name to the later functor, which will play an important role for Landau-Ginzburg models.

\begin{definition}
The TQFT
\[ Z_G^{\FF_q} = \mathcal{Q}^\# \circ (-)(\FF_q) \circ \mathcal{F}_{G}: \Bord_n \to \Vect_\CC \]
will be called the \emph{TQFT of $\FF_q$-closed points}.
\end{definition}

Note that, regardless of whether $G$ is connected, there always is a natural transformation $Z_G \Rightarrow Z_G^{\FF_q}$ due to Proposition \ref{prop:comparisontqft_arith}. When $G$ is connected, we have a natural isomorphism $Z_G^{\FF_q} \cong Z_{G(\FF_q)}^{\#}$ ($\cong Z_{G(\FF_q)}^{\Fr}$).

\section{Landau--Ginzburg models TQFT}
\label{sec:tqftLG}

In this section, we give a different point of view on the natural transformations given in Corollary \ref{cor:natural_transformation_geometric_arithmetic} by generalizing Theorem \ref{thm:character_stack_TQFT} to Landau--Ginzburg models. The generalization is relatively straightforward, we will only sketch the key steps. We begin with the definition of Landau--Ginzburg models.

\subsection{Landau--Ginzburg models}

Let $\mathfrak{S}$ be a finite type algebraic stack over a field $k$ (with affine stabilizers). A Landau--Ginzburg (LG) model over $\mathfrak{S}$ is an algebraic stack $\mathfrak{X}$ (with affine stabilizers) over $\mathfrak{S}$ equipped with a function called the potential. In this paper, we will consider two cases, namely (1) when the potential is algebraic (see Definition \ref{def:lgmodel}), and (2) when the potential can be any function on the closed points (see Definition \ref{def:lgmodel2}). Until Section \ref{sec:LGarith}, we will use the former definition, however, the results of this section hold for both definitions.

\begin{definition}\label{def:lgmodel}
The category $\LG_\mathfrak{S}$ of \emph{Landau--Ginzburg models} over $\mathfrak{S}$ is the category
\begin{itemize} 
\item whose objects are pairs $(\mathfrak{X}\rightarrow \mathfrak{S},f)$ where $\mathfrak{X}$ is an algebraic stack over $\mathfrak{S}$ of finite type (with affine stabilizers) and $f:\mathfrak{X}\to \AA^1_k$ is a map of algebraic stacks, and
\item morphisms between the objects $(\mathfrak{X}\to \mathfrak{S}, f)$ and $(\mathfrak{Y}\to \mathfrak{S}, g)$ are morphisms $\pi:\mathfrak{X}\to \mathfrak{Y}$ commuting with the structure maps (a) $\mathfrak{X}\to \mathfrak{S}$ and $\mathfrak{Y}\to \mathfrak{S}$, and (b) $\mathfrak{X}\to \AA^1_k$ and $\mathfrak{Y}\to \AA^1_k$.
\end{itemize}
\end{definition}

The definition above yields the definition of the Grothendieck ring of Landau--Ginzburg models.

\begin{definition}
    Let $\mathfrak{S}$ be an algebraic stack of finite type over a field $k$ (with affine stabilizers). The \emph{Grothendieck ring of Landau--Ginzburg models} over $\mathfrak{S}$, denoted $\K(\LG_\mathfrak{S})$, is defined as the quotient of the free abelian group on the set of isomorphism classes of Landau--Ginzburg models over $\mathfrak{S}$, by relations of the form
    \[ [(\mathfrak{X}\to \mathfrak{S}, f)] = [(\mathfrak{Z}\to \mathfrak{S}, f|_{\mathfrak{Z}})] + [(\mathfrak{U}\to \mathfrak{S}, f|_{\mathfrak{U}})] \]
    where $\mathfrak{Z}$ is a closed substack of $\mathfrak{X}$ and $\mathfrak{U}$ is its open complement equipped with the restrictions of the regular functions $f$. 
\end{definition}

Note, that the Grothendieck ring of Landau--Ginzburg models is indeed a ring, the multiplication is induced by the fibre product
\[ [(\mathfrak{X} \to \mathfrak{S}, f)] \cdot [(\mathfrak{Y}\to \mathfrak{S}, g)] = [(\mathfrak{X} \times_{\mathfrak{S}} \mathfrak{Y})\to \mathfrak{S}, h)].\]
Here $h$ is defined as 
the sequence of maps
\[h: (\mathfrak{X} \times_{\mathfrak{S}} \mathfrak{Y})\xrightarrow{(f,g)} \AA^1_k\times \AA^1_k\xrightarrow{\cdot} \AA^1_k\]
where the last map uses the multiplicative structure on $\AA^1_k$. This multiplication is indeed associative and commutative with the identity element being $[(\mathfrak{S}\xrightarrow{\id} \mathfrak{S}, 1)]$.

The following lemma relates the rings $\K(\LG_S)$ and $\K(\Stck_S)$ and it is straightforward to prove.

\begin{lemma}\label{lem:forgetfulLG}
The constant function map
\[\iota:\K(\Stck_\mathfrak{S})\rightarrow \K(\LG_\mathfrak{S})\]
induced by the map on virtual classes $[\mathfrak{X}\to \mathfrak{S}]\mapsto [(\mathfrak{X}\to \mathfrak{S}, 1)]$ is a ring homomorphism, making $\K(\LG_\mathfrak{S})$ a $\K(\Stck_\mathfrak{S})$-algebra.

Similarly, the forgetful map
\[\stI:\K(\LG_\mathfrak{S})\to \K(\Stck_\mathfrak{S})\]
induced by the map on virtual classes $[(\mathfrak{X}\to \mathfrak{S}, f)]\mapsto [\mathfrak{X}\to \mathfrak{S}]$ is a ring homomorphism as well.\qed
\end{lemma}

As in Section \ref{sec:groth}, a morphism of stacks $\mathfrak{X}\to \mathfrak{S}$ induces a $\K(\LG_\mathfrak{S})$-module structure on $\K(\LG_\mathfrak{X})$ defined on the classes as follows. Let $[(\mathfrak{Y}\to \mathfrak{X}, f:\mathfrak{Y}\to \AA^1_k)]$, $[(\mathfrak{T}\to \mathfrak{S}, f:\mathfrak{T}\to \AA^1_k)]$ be two classes, then we define the action as
\[[(\mathfrak{T}\to \mathfrak{S}, f:\mathfrak{T}\to \AA^1_k)]\cdot [(\mathfrak{Y}\to \mathfrak{X}, g:\mathfrak{Y}\to \AA^1_k)]:=[(\mathfrak{Y}\times_\mathfrak{X} (\mathfrak{T}\times_\mathfrak{S} \mathfrak{X}), h:(\mathfrak{Y}\times_\mathfrak{X} (\mathfrak{T}\times_\mathfrak{S} \mathfrak{X})\to \AA^1_k]\]
where $h=\tilde{f}\cdot\tilde{g}$ illustrated in the diagram below. 
\[ \begin{tikzcd}[column sep=3em, row sep=3em]
    \mathfrak{Y}\times_\mathfrak{X} (\mathfrak{T}\times_\mathfrak{S} \mathfrak{X}) \arrow[dd, bend right=60, "\tilde{g}"] \arrow[rrr, bend left=60, "\tilde{f}"] \arrow[r] \arrow[d] & \mathfrak{T}\times_\mathfrak{S} \mathfrak{X} \arrow[r] \arrow[d] & \mathfrak{T} \arrow[r, "f"] \arrow[d] & \mathbb{A}^1_k \\
    \mathfrak{Y} \arrow[r] \arrow[d, "g"] & \mathfrak{X} \arrow[r] & \mathfrak{S} & \\
    \mathbb{A}^1_k & & &
\end{tikzcd} \]

It is easy to see that this action extends to a $\K(\LG_\mathfrak{S})$-module structure on $\K(\LG_\mathfrak{X})$.

Furthermore, any morphism of stacks $a : \mathfrak{X} \to \mathfrak{Y}$ over $\mathfrak{S}$ yields a functor
\begin{align*}
     \LG_\mathfrak{X} &\to \LG_\mathfrak{Y} \\
    (\mathfrak{Z} \xrightarrow{h} \mathfrak{X}, f:\mathfrak{Z}\to \AA^1_k) & \mapsto (\mathfrak{Z}\xrightarrow {a \circ h} \mathfrak{Y}, f:\mathfrak{Z}\to \AA^1_k)
\end{align*}
inducing a $\K(\LG_\mathfrak{S})$-module map $a_!: \K(\LG_\mathfrak{X})\to \K(\LG_\mathfrak{Y})$. Note that this map will in general not be a ring morphism.

Similarly, pulling back along $a$ yields a functor
\[\LG_\mathfrak{Y} \to \LG_\mathfrak{X} \]
sending $(\mathfrak{Z}\to \mathfrak{Y}, f:\mathfrak{Z}\to \AA^1_k)$ to $(\mathfrak{Z} \times_\mathfrak{Y} \mathfrak{X}\to \mathfrak{X}, \tilde{f})$ illustrated in the diagram below. 
\[ \begin{tikzcd}
    \mathfrak{Z} \times_\mathfrak{Y} \mathfrak{X} \arrow{r} \arrow{d} \arrow[bend left=20]{rr}{\overline{f}} & \mathfrak{Z} \arrow[swap]{r}{f} \arrow{d} & \AA^1_k \\ \mathfrak{X} \arrow{r}{a} & \mathfrak{Y} &
\end{tikzcd} \]
This induces a ring homomorphism
\[ a^* : \K(\LG_\mathfrak{Y}) \to \K(\LG_\mathfrak{X}) , \]
 which is also a $\K(\LG_\mathfrak{S})$-algebra map.

\subsection{Landau--Ginzburg TQFT}
We alter the construction of Section \ref{sec:field_theory_and_quantization} to establish a TQFT in values of Landau--Ginzburg models computing the virtual class of character stacks. Namely, the symmetric monoidal functor called \emph{field theory} 
\[ \mathcal{F} : \Bord_n \to \Span(\Stck_k)\]
is unchanged, defined as in Section \ref{sec:field_theory_and_quantization}. However, the \emph{quantization functor}
\[ \mathcal{Q}^\LG : \Span(\Stck_\BG) \to \K(\LG_k)\textup{-}\Mod \]
is altered accordingly by assigning to an object $\mathfrak{X}$ the $\K(\LG_k)$-module $\K(\LG_\mathfrak{X})$, and to a span $\mathfrak{X} \xleftarrow{f} \mathfrak{Z} \xrightarrow{g} \mathfrak{Y}$ the morphism $g_! \circ f^* : \K(\LG_\mathfrak{X}) \to \K(\LG_\mathfrak{Y})$ of $\K(\LG_k)$-modules. As before, the quantization functor is not monoidal, however, it is symmetric and lax monoidal. The composition of the field theory and the quantization functor defines the TQFT in LG-models
\[ Z^\LG_G = \mathcal{Q}^\LG \circ \mathcal{F} : \Bord_n \to \K(\LG_k)\textup{-}\Mod . \]

This lax monoidal functor assigns to the empty set $\varnothing$ the ring $\K(\LG_k)$ and to a closed connected manifold $W$ of dimension $n$ thought of as a bordism $W : \varnothing \to \varnothing$ the ring endomorphism of $\K(\LG_k)$ that is multiplication by $[(\mathfrak{X}_G(W), \mathfrak{X}_G(W) \xrightarrow{1} \mathbb{A}_k^1)]$ yielding the following theorem.

\begin{theorem}
    \label{thm:tqft_LG}
    Let $G$ be an algebraic group. There exists a lax TQFT
    \[ Z^\LG_G : \Bord_n \to \K(\LG_k)\textup{-}\Mod \]
    computing the virtual classes of $G$-character stacks. \qed
\end{theorem}

\subsubsection{Geometric method via LG-models}\label{sec:geommet}
In this short section, we show that the TQFT $Z_G$ defined in Theorem \ref{thm:character_stack_TQFT} is an instance of the Landau-Ginzburg TQFT.

Let $\mathfrak{S}$ be an algebraic stack of finite type over $k$. Consider the forgetful map
\[\stI: \K(\LG_\mathfrak{S})\rightarrow \K(\Stck_\mathfrak{S})
\]
sending the class of a Landau-Ginzburg model $[(\mathfrak{X}\to \mathfrak{S}, f:\mathfrak{X}\to \AA^1)]$ to the virtual class $[\mathfrak{X}\to \mathfrak{S}]\in \K(\RStck_\mathfrak{S})$ (see Lemma \ref{lem:forgetfulLG}). The following statement is immediate because of the forgetful nature of $\stI$.

\begin{theorem}
    \label{thm:geommet}
    The forgetful map induces a natural transformation $\stI:Z_G^{\LG} \Rightarrow Z_G$ as functors $\Bord_n\to \Ab$. In particular, if $W$ is a closed manifold, seen as a bordism $\varnothing\rightarrow \varnothing$, then 
    \[ \stI(Z_G^{\LG}(W))(1) = Z_G(W)(1) = [\mathfrak{X}_G(W)] \]
    in $\K(\Stck_k)$. \qed
\end{theorem}

\subsubsection{Arithmetic method via LG-models}\label{sec:LGarith}

In this section, we compare the Landau-Ginzburg TQFT with the arithmetic TQFT. We first alter the definition of a Landau-Ginzburg model (LG-model for short) as follows (see Definition \ref{def:lgmodel}).

\begin{definition}\label{def:lgmodel2}
The category $\LG_\mathfrak{S}$ of Landau-Ginzburg models over an algebraic stack $\mathfrak{S}$ of finite type over a finite field $k$ (with affine stabilizers) is the category
\begin{itemize} 
\item whose objects are pairs $(\mathfrak{X}\rightarrow \mathfrak{S},f)$ where $\mathfrak{X}$ is afinite type algebraic stack over $\mathfrak{S}$ (with finite stabilizers) and $f:\mathfrak{X}(k)\to \CC$ is any $\CC$-valued function on the $k$-points of $\mathfrak{X}$, and
\item morphisms between the objects $(\mathfrak{X}\to \mathfrak{S}, f)$ and $(\mathfrak{Y}\to \mathfrak{S}, g)$ are algebraic morphisms $\pi:\mathfrak{X}\to \mathfrak{Y}$ commuting with the structure maps (a) $\mathfrak{X}\to \mathfrak{S}$ and $\mathfrak{Y}\to \mathfrak{S}$, and the maps on $k$-points (b) $f:\mathfrak{X}(k)\to \CC$ and $g:\mathfrak{Y}(k)\to \CC$.
\end{itemize}

\end{definition}

We consider the integration over the fibres map 
\[\int:\LG_\mathfrak{X}\to \CC^{\mathfrak{X}(k)}\]
sending a Landau-Ginzburg model $(\mathfrak{Y}\xrightarrow{\pi} \mathfrak{X}, f:\mathfrak{Y}(k)\to \CC)$ to $\left(x\mapsto \sum_{[y]\in [\pi^{-1}(x)]}\frac{f(y)}{|\Aut (y)|}\right)$.
It is easy to see that in the case of $\mathfrak{S}=[S/G]$ this provides a map $\int:K_0(\LG_\mathfrak{S})\rightarrow \CC^{\mathfrak{S}(k)}$ (the integration map) from the Grothendieck ring of Landau-Ginzburg models to the complex vector space of complex functions on the objects of the character groupoid $\mathfrak{S}(k)$. 

We have the following statement, completely similar to Proposition \ref{prop:comparisontqft_arith}.
\begin{proposition}
     Let $\FF_q$ be a finite field over $k$. There is a natural transformation
    \[\int: \mathcal{Q}^\LG \Rightarrow \mathcal{Q}^\# \circ (-)(\FF_q).\]
    In particular, this induces a natural transformation of TQFTs
    \[ Z_G^\LG \Rightarrow Z_G^{\FF_q} . \]
    As a consequence, if $W$ is a closed manifold, seen as a bordism $\varnothing\rightarrow \varnothing$, then 
    \[ \int Z_G^{\LG}(W)(1) = Z^{\FF_q}_G(W)(1) = |\mathfrak{X}_G(W)(\FF_q)|.\]
    \qed
\end{proposition}

Combining Proposition \ref{prop:compGconn} and the proposition above, we see that if $G$ is connected, then there is a natural transformation $Z_G^\LG\Rightarrow Z^\#_{G(\FF_q)}$ comparing the LG-valued TQFT with the arithmetic one.

\section{Applications}
\label{sec:app}

In the previous section, we provided a comparison between the TQFTs considered in this paper, in terms of natural transformations, summarized in diagram \eqref{eq:diagram_natural_transformations}. We refer to this as the \emph{arithmetic-geometric correspondence}: the TQFT $Z^\#_G$ has an arithmetic flavour as it counts solutions to algebraic equations on finite fields, whereas the TQFT $Z_G$ has an intrinsic geometric nature, in the sense that it computes a subtle algebraic invariant of the character stack.

In this section, we analyze two types of implications of the correspondence. In one implication, we use the geometric TQFT $Z_G$ in order to extract information about the character table of the algebraic group $G$ over finite fields. For example, the first column of the character table, consisting of the dimensions of the irreducible characters and their multiplicities, can easily be determined, but also some other sums of columns.

In the other implication, we attempt to lift the eigenvectors of the arithmetic TQFT $Z^\#_G$ to eigenvectors of the geometric TQFT $Z_G$. For $G$ the group of upper triangular matrices (of small rank), we show there exist canonical such lifts and, moreover, that these lifts immensely simplify the computations for $Z^\#_G$, as performed in \cite{hablicsek2022virtual}. That is, in \cite{hablicsek2022virtual}, computed-assisted computations were performed on a submodule with $16$ generators, while with the simplification, only $2$ or $3$ generators are needed.

The bridge between these two worlds is Corollary \ref{cor:relation_eigenvalues_geometric_arithmetic}, which follows directly from the arithmetic-geometric correspondence \eqref{eq:diagram_natural_transformations}. To state it properly, recall that $Z_G\left(\bdgenus[0.7]\right)$ is an endomorphism of the $\K(\Stck_R)$-module $\K(\Stck_{[G/G]})$. However, suppose that the $\K(\Stck_R)$-submodule
$$
    \mathcal{V} = \langle Z_G\left(\bdgenus[0.7]\right)^g(\textup{\bfseries 1}_G) \textup{ for } g \in \ZZ_{\ge 0} \rangle \subseteq \K(\Stck_{[G/G]})
$$
is finitely generated. Then $Z_G\left(\bdgenus[0.7]\right)$ is an endomorphism of $\mathcal{V}$. Thus, if $x_1, \ldots, x_n \in \mathcal{V}$ is a set of generators, we have that for all $i = 1, \ldots, n$ we can write
$$
    Z_G\left(\bdgenus[0.7]\right)(x_i) = \sum_{j=1}^n a_{ij} x_j, \quad \textrm{with } a_{ij} \in \K(\Stck_R).
$$
With this information, in analogy with the vector space case, we can form the matrix $A = (a_{ij})$ of coefficients representing $Z_G\left(\bdgenus[0.7]\right)$. Note that $A$ is not necessarily unique. With this matrix, given a ring morphism $R \to \FF_q$, using the construction of Definition \ref{defn:integral-fibers}, we can form the $n \times n$ matrix of complex numbers
\[ \int_{R} A = \left(\int_{R} a_{ij}\right)_{ij} \quad \textup{ where } \quad \int_{R} a_{ij} \in \CC^{R(\FF_q)} = \CC. \]
With this notation, the main result of this section is the following.

\begin{corollary}
    \label{cor:relation_eigenvalues_geometric_arithmetic}
    Let $G$ be a connected algebraic group over a finitely generated $\ZZ$-algebra $R$, and let $\FF_q$ be a finite field extending $R$. Denote by $\textup{\bfseries 1}_G \in \K(\Stck_{[G/G]})$ the class of the inclusion of the identity in $G$.
    If the $\K(\Stck_R)$-module $\mathcal{V} = \langle Z_G\left(\bdgenus[0.7]\right)^g(\textup{\bfseries 1}_G) \textup{ for } g \in \ZZ_{\geq 0} \rangle$ is finitely generated, then:
    \begin{enumerate}[(i)]
        \item The submodule $\int_{[G/G]} \mathcal{V}$ of $\CC^{[G/G](\FF_q)}$, both seen as $\K(\Stck_R)$-modules, is generated by the sums of equidimensional irreducible characters.
        \item The dimensions of the complex irreducible characters of $G(\FF_q)$ are precisely given by
        \[ d_i = \frac{|G(\FF_q)|}{\sqrt{\lambda_i}} \]
        for $\lambda_i \in \ZZ$ the eigenvalues of $\int_R A$, where $A$ is any matrix representing the linear map $Z_G\left(\bdgenus[0.7]\right)$ with respect to a generating set of $\mathcal{V}$.
        \item Write $\int_R \textup{\bfseries 1}_G = \sum_i v_i$, where $v_i$ are eigenvectors of $\int_R A$ with eigenvalues $\lambda_i$. Then each $v_i$ is a scalar multiple of the sum of equidimensional characters
        \[ v_i = \frac{d_i}{|G(\FF_q)|}\sum_{\substack{\chi \in \hat{G}\textup{\ s.t.}  \\ \chi(1) = d_i}} \chi .\]
    \end{enumerate}
\end{corollary}
\begin{proof}
\textit{(i)} Since $G$ is connected, by Corollary \ref{cor:natural_transformation_geometric_arithmetic}, we get a commutative square
\[ \begin{tikzcd}[column sep=5em]
    \K(\Stck_{[G/G]}) \arrow{r} \arrow[swap]{d}{\int_{[G/G]}} & \K(\Stck_{[G/G]}) \arrow{d}{\int_{[G/G]}} \\
    \CC^{[G(\FF_q)/G(\FF_q)]} \arrow{r} & \CC^{[G(\FF_q)/G(\FF_q)]}
\end{tikzcd} \]
where the top and bottom map are given by $Z_G\left(\bdgenus[0.7]\right)$ and $Z^\#_{G(\FF_q)}\left(\bdgenus[0.7]\right)$, respectively. Also, we see $\CC^{[G/G](\FF_q)} = R_\CC(G(\FF_q))$ as a $\K(\Stck_{R})$-module through restriction of scalars via $\int_{R}$. Now, observe that $\int_{[G/G]}\textup{\bfseries 1}_G \in R_\CC(G(\FF_q))$ is the delta-function over the identity of $ G(\FF_q)$, so we can write
$$
    \int_{[G/G]}\textup{\bfseries 1}_G=\frac{1}{|G(\FF_q)|}\sum_{\chi\in \hat{G}} \chi(1)\chi.
$$
More in general, using that the irreducible characters are eigenvectors of $Z^\#_{G(\FF_q)}\left(\bdgenus[0.7]\right)$ with eigenvalue $|G(\FF_q)|^2/\chi(1)^2$, we have
$$
    Z^\#_{G(\FF_q)}\left(\bdgenus[0.7]\right)^g \left(\int_{[G/G]}\textup{\bfseries 1}_G \right) =\frac{1}{|G(\FF_q)|}\sum_{\chi\in \hat{G}} \frac{|G|^{2g}}{\chi(1)^{2g-1}}\chi(1)\chi,
$$
which shows that the space generated by the vectors $Z^\#_{G(\FF_q)}\left(\bdgenus[0.7]\right)^g \left(\int_{[G/G]}\textup{\bfseries 1}_G \right) \in R_\CC(G(\FF_q))$ for $g = 0, 1, \ldots$ is generated by the sums $\sum_{\chi(1) = n} \chi$ of irreducible characters of dimension $n \geq 1$. But, by the naturality of $\int_{[G/G]}$, this space is exactly $\int_{[G/G]} \mathcal{V}$.

    \textit{(ii)}
By naturality, $Z^\#_{G(\FF_q)}\left(\bdgenus[0.7]\right)$ is an endomorphism of the submodule $\int_{[G/G]} \mathcal{V} \subseteq  R_\CC(G(\FF_q))$. Furthermore, we see that $\int_{[G/G]} A$ represents this restriction written in the set of generators $\int_{[G/G]} x_i$ of $R_\CC(G(\FF_q))$, where $x_1, \ldots, x_n$ are generators of $\mathcal{V}$. Hence, the non-zero eigenvalues and eigenvectors of this matrix are also eigenvalues and eigenvectors of $Z^\#_{G(\FF_q)}\left(\bdgenus[0.7]\right)$, which are of the form $\frac{|G(\FF_q)|^2}{\chi(1)^2}$. Since all these dimensions appear in one of the generators of $\int_{[G/G]}$, this proves the statement.

    \textit{(iii)} We can write $\int_{[G/G]}\textup{\bfseries 1}_G$ in two ways: first as $\frac{1}{|G(\FF_q)|}\sum_{\chi\in \hat{G}} \chi(1)\chi$, second as $\sum_i v_i$ where $v_i$ are the eigenvalues of $\int_R A$. The statement follows from this.
\end{proof}

\begin{remark}
    \label{rem:nonconnected}
    In a previous paper \cite{gonzalez2022virtual}, the authors of this paper considered a different computational framework to investigate the virtual classes of character stacks. In that framework, additional basepoints were chosen on the manifolds and the bordisms to make the field theory functorial. We note that the main result of this paper, Theorem \ref{introthm:nattrans}, also holds with some natural alterations in the case of that framework. This allows us to compare the endomorphism $Z_G\left(\bdgenus[0.7]\right)$ of the geometric TQFT to the character table of $G$ when $G$ is not necessarily connected.
\end{remark}

\subsection{From geometry to characters}

Let us illustrate how the geometry of $Z_G$ implies, through the arithmetic-geometric correspondence, information on the arithmetic side: the representation theory of $G$ over finite fields $\FF_q$. As a simple example, we take $G$ to be the affine linear group of rank 1,
\[ \AGL_1(k) = \left\{ \begin{pmatrix} a & b \\ 0 & 1 \end{pmatrix} : a \ne 0 \right\} . \]
We introduce the following two locally closed subvarieties of $\AGL_1$:
\[ I = \left\{ \begin{pmatrix} 1 & 0 \\ 0 & 1 \end{pmatrix}\right\}, \qquad J = \left\{ \begin{pmatrix} 1 & b \\ 0 & 1 \end{pmatrix} : b \ne 0 \right\} \]
whose natural inclusions $I \to G$ and $J \to G$ induce elements $[I/G], [J/G] \in \K(\Stck_{[G/G]})$. The fact that the elements of the commutator subgroup have ones on the diagonal translates into the fact that the $\K(\Stck_k)$-submodule generated by $[I/G]$ and $[J/G]$ will be invariant under $Z_G\left(\bdgenus[0.7]\right)$.

Denoting $q = [\AA^1_k]$, one easily shows (as in \cite{gonzalez2022virtual} or \cite{gologmun05}) that
\[ Z_G\left(\bdgenus\right) = \left(
\begin{array}{cc}
    q^2 (q - 1) & q^2 (q - 2) (q - 1) \\
    q^2 (q - 2) & q^2 (q^2 - 3 q + 3)
\end{array} \right) \]
with respect to $[I/G]$ and $[J/G]$. This matrix can be diagonalized with eigenvectors $(q - 1)([I/G] + [J/G])$ and $(q - 1)[I/G] - [J/G]$ and eigenvalues $q^2 (q - 1)^2$ and $q^2$, respectively.

As the computations for $Z_G\left(\bdgenus[0.7]\right)$ can be performed over $\ZZ$, they are also valid over any finite field $k = \FF_q$. Under the natural transformation of Corollary \ref{cor:natural_transformation_geometric_arithmetic}, the eigenvectors transform to the following functions on $\AGL_1(\FF_q)$: 
\[\left(\int_{[G/G]} (q - 1) ([I/G] + [J/G]) \right)(g) = \begin{cases} q - 1 & \textup{ if } g = \left(\begin{smallmatrix} 1 & b \\ 0 & 1 \end{smallmatrix}\right) , \\ 0 & \textup{ otherwise} , \end{cases}\]
and
\[ \left(\int_{[G/G]} (q - 1) [I/G] - [J/G]\right)(g) = \begin{cases} (q - 1) & \textup{ if } g = 1, \\ -1 & \textup{ if } g \in J, \\ 0 & \text{otherwise} . \end{cases} \]
From Section \ref{sec:representation_ring} we know that the eigenvalues, which are $q^2 (q - 1)^2$ and $q^2$, correspond to the values $\left(\frac{|\AGL_1(\FF_q)|}{\chi(1)}\right)^2$ for the irreducible characters $\chi$. Hence, the irreducible characters are of dimension $1$ or $q - 1$. The above eigenvectors are already properly normalized, meaning $\eta(1) = \frac{1}{|G|} (v_1 + (q - 1) v_{q - 1})$. Therefore, the first function is the sum of the $(q - 1)$ $1$-dimensional irreducible characters of $\AGL_1(\FF_q)$ and the second function is the character of the unique $(q - 1)$-dimensional representation of $\AGL_1(\FF_q)$. 
Note that in the special case of $q = 2$, there are 2 irreducible characters both of dimension $1$, which are precisely given by the two above functions.

\subsection{From characters to geometry}

Let us illustrate how the arithmetic information, i.e.\ the representation theory of $G$ over $\FF_q$, provides geometric insight into the TQFT $Z_G$.
We will show how the representation theory of the groups of unipotent upper triangular matrices over $\FF_q$ can be used to simplify the corresponding geometric TQFT $Z_G$, which was also considered and computed in \cite{hablicsek2022virtual}. In particular, we provide a new smaller set of generators for this TQFT motivated by the arithmetic side. More precisely, we obtain this new generating set by canonically lifting the sums of equidimensional characters to the Grothendieck ring of varieties. These generators will be given by classes of locally closed subvarieties of $G$.

\begin{example}[Unipotent $3 \times 3$ matrices]
Consider the group $\UU_3(\FF_q)$ of unipotent matrices of rank $3$ over a finite field $\FF_q$,
\begin{equation} \label{eq:presentation_U3} \UU_3(\FF_q) = \left\{ \begin{pmatrix} 1 & x & y \\ 0 & 1 & z \\ 0 & 0 & 1 \end{pmatrix} : x, y, z \in \FF_q \right\} . \end{equation}
The irreducible complex characters of $\UU_3(\FF_q)$ are of dimension $1$ or $q$. Denote the set of $1$-dimensional characters by $X_1$ and of the $q$-dimensional characters by $X_q$. Summing the $1$-dimensional characters, we find that
\[ v_1 = \sum_{\chi \in X_1} \chi \quad \textup{ is given by } \quad v_1 \begin{pmatrix} 1 & x & y \\ 0 & 1 & z \\ 0 & 0 & 1 \end{pmatrix} = \begin{cases} q & \textup{ if } x = z = 0 , \\ 0 & \textup{ otherwise} , \end{cases} \]
and summing the $q$-dimensional characters, we find that
\[ v_2 = \sum_{\chi \in X_q} \chi \quad \textup{ is given by } \quad v_2 \begin{pmatrix} 1 & x & y \\ 0 & 1 & z \\ 0 & 0 & 1 \end{pmatrix} = \begin{cases} -q & \textup{ if } x = z = 0 \textup{ and } y \ne 0, \\ q(q - 1) & \textup{ if } x = y = z = 0, \\ 0 & \textup{ otherwise}. \end{cases} \]
These eigenvectors can be canonically lifted to the elements in $\K(\Stck_{[G/G]})$ for $G = \UU_3$ given by
\[ q [\{ x = z = 0 \}] \quad \textup{ and } \quad -q [\{ x = z = 0, \; y \ne 0 \}] + q (q - 1) [\{ x = y = z = 0 \}] , \]
in terms of classes of $G$-equivariant subvarieties of $\UU_3$, where $x, y, z$ are as in the presentation \eqref{eq:presentation_U3}.
Indeed, from the computations of \cite{hablicsek2022virtual} it can be seen that these elements are eigenvectors of $Z_G$. The eigenvalues are $q^6$ and $q^4$, respectively, where now $q = [\mathbb{A}^1_k]$.
In fact, the submodule of $\K(\Stck_{[G/G]})$ generated by these two eigenvectors contains $Z_G(\bdunit[0.9])(1)$ and is invariant under $Z_G(\bdgenus[0.7])$, and is therefore a simplification of the submodule used in \cite{hablicsek2022virtual}, which used $5$ generators.
\end{example}

\begin{example}[Unipotent $4 \times 4$ matrices]
Consider the group $\UU_4(\FF_q)$ of unipotent matrices of rank $4$ over a finite field $\FF_q$. This group has three families of irreducible complex characters: the $1$-dimensional characters $X_1$, the $q$-dimensional characters $X_q$ and the $q^2$-dimensional characters $X_{q^2}$. Summing characters of the same dimension, we find
\begin{align*}
    \sum_{\chi \in X_1} \chi \begin{pmatrix} 1 & a & b & c \\ 0 & 1 & d & e \\ 0 & 0 & 1 & f \\ 0 & 0 & 0 & 1 \end{pmatrix} &= \begin{cases} q^2 & \textup{ if } a = d = f = 0, \\ 0 & \textup{ otherwise} , \end{cases} \\
    \sum_{\chi \in X_q} \chi \begin{pmatrix} 1 & a & b & c \\ 0 & 1 & d & e \\ 0 & 0 & 1 & f \\ 0 & 0 & 0 & 1 \end{pmatrix} &= \begin{cases} q^4 & \textup{ if } a = b = d = e = f = 0, \\ -q^2 & \textup{ if } a = d = f\\ 0 & \textup{ otherwise} , \end{cases} \\
    \sum_{\chi \in X_{q^2}} \chi \begin{pmatrix} 1 & a & b & c \\ 0 & 1 & d & e \\ 0 & 0 & 1 & f \\ 0 & 0 & 0 & 1 \end{pmatrix} &= \begin{cases} q^3 (q - 1) & \textup{ if } a = b = c = d = e = f = 0, \\ -q^3 & \textup{ if } a = b = d = e = f = 0 \textup{ and } c \ne 0 , \\ 0 & \textup{ otherwise} . \end{cases}
\end{align*}
Again, we can canonically lift these functions to elements in $\K(\Stck_{[G/G]})$ for $G = \UU_4$ given by
\begin{align*}
    & q^2 [\{ a = d = f = 0 \}], \\
    & q^4 [\{ a = b = d = e = f = 0 \}] - q^2 [\{ a = d = f = 0 \}] , \\
    \textup{ and } \quad & q^3 (q - 1) [\{ a = b = c = d = e = f = 0 \}] - q^3 [\{ a = b = d = e = f = 0 \textup{ and } c \ne 0 \}] .
\end{align*}
From the computations of \cite{hablicsek2022virtual} it can be seen that these elements are indeed eigenvectors of $Z_G$.
As before, we can use these three eigenvectors to generate a submodule of $\K(\Stck_{[G/G]})$ replacing the one from \cite{hablicsek2022virtual}, which used $16$ generators, a significant simplification.
\end{example}

\begin{example}[$\GG_m \rtimes \ZZ/2\ZZ$]\label{ex:gmz2a}
A very interesting phenomenon arises in the case of the group $G = \GG_m \rtimes \ZZ/2\ZZ$ where the group $\ZZ/2\ZZ$ acts on $\GG_m$ by negation, namely, that it is impossible that the geometric TQFT consists of generators that are subvarieties of $G$! Indeed, consider the group over a finite field $\FF_q$, and assume for simplicity that $q$ is odd. Using Remark \ref{rem:nonconnected}, we can study the geometric TQFT from the character table. The sums of equidimensional characters can be expressed as
\[ {\arraycolsep=1em \def\arraystretch{1.5}
\begin{array}{c|ccc} 
    & \{ 1 \} & \{ t \in G(\FF_q) \mid t \textup{ is a square} \} & \{ t \in G(\FF_q) \mid t \textup{ is not a square} \}  \\ \hline 
    v_1 =\sum_{\varepsilon, \delta} \rho_{\varepsilon, \delta} & 4 & 4 & 0 \\
    v_2=\sum_k \tau_k & q - 3 & -2 & 0 
\end{array}} \]
In this case, lifting the eigenvectors $v_1$ and $v_2$ to elements of the Grothendieck ring is slightly more involved: there exists no `subvariety of squares' of $G$. However, we can instead consider the variety $X = \GG_m$ over $\GG_m$ given by $x \mapsto x^2$. Counting the fibers of $X \to \GG_m$ over $\FF_q$, we find
\[ | X(\FF_q)|_t | = \left\{ \begin{array}{cl} 2 & \textup{ if $t$ is a square}, \\ 0 & \textup{ if $t$ is not a square} . \end{array}\right. \]
This suggests that eigenvectors of $Z_G$ over $R = \ZZ[\frac{1}{2}]$ are given by
\[ 2 [X] \quad \textup{ and } \quad (q - 1) [\{ 1 \}] - [X] . \]
Indeed, these elements agree with the generators for the submodule of $\K(\Stck_{[G/G]})$ as used for the computations in \cite{gonzalez2022virtual}.
\end{example}

\bibliographystyle{plain}
\addcontentsline{toc}{section}{References}
\bibliography{bibliography}

\begin{thebibliography}{10}

\bibitem{abrams1996two}
Lowell Abrams.
\newblock Two-dimensional topological quantum field theories and {F}robenius
  algebras.
\newblock {\em Journal of Knot theory and its ramifications}, 5(05):569--587,
  1996.

\bibitem{atiyah1988topological}
Michael Atiyah.
\newblock Topological quantum field theory.
\newblock {\em Publications Math{\'e}matiques de l'IH{\'E}S}, 68:175--186,
  1988.

\bibitem{bridger2022character}
Nick Bridger and Masoud Kamgarpour.
\newblock Character stacks are {PORC} count.
\newblock {\em Journal of the Australian Mathematical Society}, pages 1--22,
  2022.

\bibitem{brown1967groupoids}
Ronald Brown.
\newblock Groupoids and van {K}ampen's theorem.
\newblock {\em Proceedings of the London Mathematical Society}, 3(3):385--401,
  1967.

\bibitem{cooper1994plane}
Daryl Cooper, Marc Culler, Henry Gillet, Darren~D Long, and Peter~B Shalen.
\newblock Plane curves associated to character varieties of 3-manifolds.
\newblock {\em Inventiones mathematicae}, 118(1):47--84, 1994.

\bibitem{cooper1998representation}
Daryl Cooper and Darren~D Long.
\newblock Representation theory and the {A}-polynomial of a knot.
\newblock {\em Chaos, Solitons \& Fractals}, 9(4-5):749--763, 1998.

\bibitem{culler1983varieties}
Marc Culler and Peter~B Shalen.
\newblock Varieties of group representations and splittings of 3-manifolds.
\newblock {\em Annals of Mathematics}, pages 109--146, 1983.

\bibitem{de2012topology}
Mark~Andrea de~Cataldo, Tam{\'a}s Hausel, and Luca Migliorini.
\newblock Topology of {H}itchin systems and {H}odge theory of character
  varieties: the case ${A}_1$.
\newblock {\em Annals of Mathematics}, pages 1329--1407, 2012.

\bibitem{de2022hitchin}
Mark~Andrea de~Cataldo, Davesh Maulik, and Junliang Shen.
\newblock Hitchin fibrations, abelian surfaces, and the ${P}={W}$ conjecture.
\newblock {\em Journal of the American Mathematical Society}, 35(3):911--953,
  2022.

\bibitem{diaconescu2018bps}
Duiliu-Emanuel Diaconescu, Ron Donagi, and Tony Pantev.
\newblock {BPS} states, torus links and wild character varieties.
\newblock {\em Communications in Mathematical Physics}, 359:1027--1078, 2018.

\bibitem{Dijkgraaf1989}
R.~Dijkgraaf.
\newblock {\em {A geometric approach to two dimensional conformal field
  theory.}}
\newblock PhD thesis, University of Utrecht, 1989.

\bibitem{dunfield2004non}
Nathan~M Dunfield and Stavros Garoufalidis.
\newblock Non-triviality of the {A}-polynomial for knots in ${S}^3$.
\newblock {\em Algebraic \& Geometric Topology}, 4(2):1145--1153, 2004.

\bibitem{eke09}
Torsten Ekedahl.
\newblock The {G}rothendieck group of algebraic stacks.
\newblock {\em arXiv e-prints}, pages arXiv--0903, 2009.

\bibitem{frobenius1896gruppencharaktere}
G~Frobenius.
\newblock {\"U}ber {G}ruppencharaktere.
\newblock {\em {S}itzungsberichte der {K}{\"o}niglich {P}reu{\ss}ischen
  {A}kademie der {W}issenschaften zu {B}erlin}, 1896.

\bibitem{fulton2013representation}
William Fulton and Joe Harris.
\newblock {\em Representation theory: a first course}, volume 129.
\newblock Springer Science \& Business Media, 2013.

\bibitem{arXiv181009714}
{\'A}ngel Gonz{\'a}lez-Prieto.
\newblock Motivic theory of representation varieties via topological quantum
  field theories.
\newblock {\em arXiv preprint arXiv:1810.09714v2 [math.AG]}, 2019.

\bibitem{gon20}
{\'A}ngel Gonz\'{a}lez-Prieto.
\newblock Virtual classes of parabolic $\text{SL}_2(\mathbb{C})$-character
  varieties.
\newblock {\em Adv. Math.}, 368:107--148, 2020.

\bibitem{gonzalez2023quantization}
{\'A}ngel Gonz{\'a}lez-Prieto.
\newblock Quantization of algebraic invariants through topological quantum
  field theories.
\newblock {\em Journal of Geometry and Physics}, 189:104849, 2023.

\bibitem{gonzalez2022virtual}
{\'A}ngel Gonz{\'a}lez-Prieto, M{\'a}rton Hablicsek, and Jesse Vogel.
\newblock Virtual classes of character stacks.
\newblock {\em arXiv preprint arXiv:2201.08699}, 2022.

\bibitem{gonzalez2023character}
{\'A}ngel Gonz{\'a}lez-Prieto and Marina Logares.
\newblock On character varieties of singular manifolds.
\newblock {\em Research in the Mathematical Sciences}, 10(3):32, 2023.

\bibitem{gologmun05}
{\'A}ngel Gonz{\'a}lez-Prieto, Marina Logares, and Vicente Mu{\~n}oz.
\newblock Representation variety for the rank one affine group.
\newblock In {\em Mathematical Analysis in Interdisciplinary Research}, pages
  381--416. Springer, 2021.

\bibitem{gonlogmun20}
{\'A}ngel Gonz\'{a}lez-Prieto, Marina Logares, and Vicente Muñoz.
\newblock A lax monoidal topological quantum field theory for representation
  varieties.
\newblock {\em B. Sci. Math.}, 161, 2020.
\newblock 102871.

\bibitem{hablicsek2022virtual}
M{\'a}rton Hablicsek and Jesse Vogel.
\newblock Virtual classes of representation varieties of upper triangular
  matrices via topological quantum field theories.
\newblock {\em SIGMA. Symmetry, Integrability and Geometry: Methods and
  Applications}, 18:095, 2022.

\bibitem{hausel2022p}
Tam\'as Hausel, Anton Mellit, Alexandre Minets, and Olivier Schiffmann.
\newblock ${P} = {W}$ via ${H}_2$.
\newblock {\em arXiv preprint arXiv:2209.05429}, 2022.

\bibitem{hausel2008mixed}
Tam{\'a}s Hausel and Fernando Rodriguez-Villegas.
\newblock Mixed {H}odge polynomials of character varieties.
\newblock {\em Inventiones Mathematicae}, 174(3):555--624, 2008.

\bibitem{hausel2003mirror}
Tam\'{a}s Hausel and Michael Thaddeus.
\newblock Mirror symmetry, {L}anglands duality, and the {H}itchin system.
\newblock {\em Invent. Math.}, 153(1):197--229, 2003.

\bibitem{kock}
Joachim Kock.
\newblock {\em Frobenius algebras and 2D topological quantum field theories}.
\newblock Cambridge University Press, Cambridge, 2003.

\bibitem{kres99}
Andrew Kresch.
\newblock Cycle groups for {A}rtin stacks.
\newblock {\em Inventiones Mathematicae}, 138(3):495--536, 1999.

\bibitem{Lang1956}
Serge Lang.
\newblock Algebraic groups over finite fields.
\newblock {\em Amer. J. Math.}, 78:555--563, 1956.

\bibitem{letellier2020series}
Emmanuel Letellier and Fernando Rodriguez-Villegas.
\newblock E-series of character varieties of non-orientable surfaces.
\newblock In {\em Annales de l'Institut Fourier}, pages 1--36, 2023.

\bibitem{logmunnew13}
Marina Logares, Vicente Munoz, and Peter Newstead.
\newblock Hodge polynomials of $\textup{SL}(2,\mathbb{C})$ -character varieties
  for curves of small genus.
\newblock {\em Revista matem{\'a}tica complutense}, 26(2):635--703, 2013.

\bibitem{mar17}
Javier Mart{\'\i}nez.
\newblock E-polynomials of $\textup{PGL}(2,\mathbb{C})$-character varieties of
  surface groups.
\newblock {\em preprint}, 2017.

\bibitem{marmun16}
Javier Mart{\'\i}nez and Vicente Mu{\~n}oz.
\newblock E-polynomials of the $\textup{SL}(2,\mathbb{C})$-character varieties
  of surface groups.
\newblock {\em International Mathematics Research Notices}, 2016(3):926--961,
  2016.

\bibitem{maulik2022p}
Davesh Maulik and Junliang Shen.
\newblock The ${P}= {W} $ conjecture for ${GL}_n$.
\newblock {\em arXiv preprint arXiv:2209.02568}, 2022.

\bibitem{mauri2021topological}
Mirko Mauri.
\newblock Topological mirror symmetry for rank two character varieties of
  surface groups.
\newblock {\em Abh. Math. Semin. Univ. Hambg.}, 91(2):297--303, 2021.

\bibitem{mauri2022geometric}
Mirko Mauri, Enrica Mazzon, and Matthew Stevenson.
\newblock On the geometric ${P}={W}$ conjecture.
\newblock {\em Selecta Mathematica}, 28(3):1--45, 2022.

\bibitem{mellit2020poincare}
Anton Mellit.
\newblock Poincar{\'e} polynomials of character varieties, {M}acdonald
  polynomials and affine springer fibers.
\newblock {\em Annals of Mathematics}, 192(1):165--228, 2020.

\bibitem{mereb2015polynomials}
Martin Mereb.
\newblock On the ${E}$-polynomials of a family of $\textrm{Sl}_n$-character
  varieties.
\newblock {\em Mathematische Annalen}, 363(3):857--892, 2015.

\bibitem{milnor}
John Milnor.
\newblock Lectures on the $h$-cobordism theorem, 1965.

\bibitem{mumford1994geometric}
David Mumford, John Fogarty, and Frances Kirwan.
\newblock {\em Geometric invariant theory}, volume~34.
\newblock Springer Science \& Business Media, 1994.

\bibitem{munoz2016geometry}
Vicente Munoz and Joan Porti.
\newblock Geometry of the ${SL}(3, \mathbb{C})$-character variety of torus
  knots.
\newblock {\em Algebraic \& Geometric Topology}, 16(1):397--426, 2016.

\bibitem{vogel2023motivic}
Jesse Vogel.
\newblock Motivic {H}igman's conjecture.
\newblock {\em arXiv preprint arXiv:2301.02439}, 2023.

\bibitem{whitehead1940c1}
John Henry~C Whitehead.
\newblock On ${C}^1$-complexes.
\newblock {\em Annals of Mathematics}, pages 809--824, 1940.

\bibitem{witten1991quantum}
Edward Witten.
\newblock On quantum gauge theories in two dimensions.
\newblock {\em Communications in Mathematical Physics}, 141(1):153--209, 1991.

\end{thebibliography}

\end{document}